\documentclass[12pt]{amsart}

\usepackage{xcolor}

\newtheorem{introtheorem}{Theorem}
\newtheorem{theorem}{Theorem}[section]
\newtheorem{corollary}[theorem]{Corollary}
\newtheorem{definition}[theorem]{Definition}
\newtheorem{lemma}[theorem]{Lemma}
\newtheorem{proposition}[theorem]{Proposition}
\newtheorem{remark}[theorem]{Remark}
\newtheorem{conjecture}[theorem]{Conjecture}

\newcommand*{\QEDB}{\null\nobreak\hfill\ensuremath{\Box}}

\DeclareMathOperator{\diam}{diam}
\DeclareMathOperator{\vol}{vol}

\usepackage[utf8]{inputenc}
\usepackage{mathtools}

\title{K{\"a}hler thresholds}
\author{Gabriella Clemente}
\email{gabriella.clemente@cnrs.fr}
\urladdr{https://sites.google.com/view/gclemente}
\address{Université Paris Cité, CNRS, IRIF, Paris, France}
\author{Carlos Simpson}
\email{carlos.simpson@univ-cotedazur.fr}
\urladdr{https://math.univ-cotedazur.fr/u/carlos/}
\address{Université C{\^o}te d'azur, CNRS, LJAD, France}
\begin{document}

\begin{abstract}
    The topology of K{\"a}hler manifolds is largely determined by the geometry due to its rigidity. In particular, the cohomology and Hodge theory of compact K{\"a}hler manifolds is quite restricted. We prove that within a certain threshold, the odd Betti numbers of any compact almost-hermitian manifold satisfying a degenerate K{\"a}hler condition are even, and the even Betti numbers are strictly positive. We call this new type of degenerated K{\"a}hler manifold acK. Our approach to proving these results makes use of a compactness theorem for acK manifolds, and a new version of Hodge theory for compact manifolds endowed with a Sobolev regular K{\"a}hler strucutre. In addition, we lay out a program to pursue the study of acK geometry that accommodates not only the classical viewpoint, but also constructive and finitary proofs, as well as formalization with proof assistants.
\end{abstract}


\noindent

\maketitle

\section*{Introduction}

This article makes a contribution to the program of extending the traditional Hodge theory of compact K{\"a}hler manifolds and its geometrical consequences to less integrable settings, i.e.\ to compact almost-symplectic almost-complex manifolds. We introduce a new kind of almost-hermitian manifold that is close to being K{\"a}hler, and prove that under certain bounded geometry assumptions, the odd Betti numbers are even. Recall that an almost-hermitian manifold $(X,g,J)$ is K{\"a}hler if $J$ is parallel w.r.t.\ the Levi-Civita connection $\nabla$ of $g$. We introduce $\epsilon$-almost-K{\"a}hler ($\epsilon$-acK) manifolds, which are defined by the condition that $\nabla J$ is bounded in norm by $\epsilon$. These manifolds are close to being K{\"a}hler in the sense that the almost-complex structure is close to being complex and the associated fundamental $2$-form is close to being symplectic. The precise relationship is given in Proposition \ref{KT_lemma}. 

Let $\mathcal{M}(n,d,k,i_0)$ be the class of compact $n$-dimensional Riemannian manifolds of diameter, Ricci curvature, and injectivity radius bounded by $d,k,i_0$, respectively. Let $a$ be a positive real number and $\mathcal{A}(n,d,k,i_0,a)$ be the space of $\epsilon$-acK manifolds $(X,g,J)$ with constant $\epsilon \in (0,a]$ such that $(X,g) \in \mathcal{M}(n,d,k,i_0)$. Note that $\mathcal{A}(n,d,k,i_0,a)$ can equivalently be described as the space of $a$-acK manifolds $(X,g,J)$ with $(X,g) \in \mathcal{M}(n,d,k,i_0)$, c.f.\ Definition \ref{KT_def2}. Indeed, if $a\leq b$, then $\mathcal{A}(n,d,k,i_0,a) \subseteq \mathcal{A}(n,d,k,i_0,b)$. Here is our main theorem.

\begin{introtheorem}\label{Betti_numbers_thm}
Given an integer $n>1$, a real number $a\geq 1$ and bounds $d,k,i_0$, there is an $0<\epsilon_0 \leq a$ such that for any Riemannian manifold $(X,g) \in \mathcal{M}(n,d,k,i_0)$, if $(X,g,J)$ is $\epsilon_0$-almost-complex-K{\"a}hler, then the odd Betti numbers of $X$ are even.
\end{introtheorem}

In other words, given $(n,a,d,k,i_0) \in \mathbb{Z}_{\geq 2} \times \mathbb{R}_{\geq 1} \times \mathbb{R}_{> 0} \times \mathbb{R}_{\geq 0} \times \mathbb{R}_{>0}$, there is an $\epsilon_0 \coloneqq \epsilon_0 (n,d,k,i_0)\leq a$ s.t.\ for any $(X,g,J) \in \mathcal{A}(n,d,k,i_0,\epsilon_0)$, we have that $b_{2k+1}(X)$ is even for all $k$.

The proof of Theorem \ref{Betti_numbers_thm} relies on a compactness result for acK manifolds (see Definition \ref{KT_def2}), and a Sobolev regular Hodge decomposition theorem, which we discuss below. 

\begin{introtheorem}\label{prepctness_thm}
For any $a > 0$ and for appropriate $p\gg 0$, the space $\mathcal{A}(n,d,k,i_0,a)$ of acK manifolds of bounded geometry is precompact in the weak-$W^{2,p} \times W^{1,p}$ topology.
\end{introtheorem}

Let $g$ be a $W^{2,p}$ Riemannian metric on a smooth manifold $X$, and $\mathcal{H}^k_g (X)$ be the space of harmonic w.r.t.\ $g$ $k$-forms. We prove that \[\mathcal{H}^k_g (X) \simeq H^k_{DR} (X),\] where $H^k_{DR} (X)$ is the $k^{th}$ complex De Rham cohomology group of $X$ (c.f.\ Corollary \ref{Hodge_thm}). Now, let $J$ be a $g$-compatible $W^{1,p}$ almost-complex structure on $X$. Suppose that $(X,g,J)$ is formally K{\"a}hler, i.e.\ $\nabla J=0$. In this case, our less regular than usual Hodge theorem together with analysis considerations gives a decomposition of $H^k_{DR} (X)$ into $W^{1,2}$ $(p,q)$-forms (see Corollary \ref{pqdecomp}). These are the main technical components of the proof of Theorem \ref{Betti_numbers_thm}.

The proof of Theorem \ref{Betti_numbers_thm} indicates that there are probably many more topological invariants of compact K{\"a}hler manifolds that are preserved within a certain radius of non-integrability. Theorem \ref{Thm-P} is a metatheorem of sorts that should be useful when it comes to identifying such invariants. We coin the term \emph{K{\"a}hler threshold}, which can be informally understood to be a set of constraints that determine a maximal category of manifolds for which basic facts of K{\"a}hler geometry continue to hold true. In this sense, the spaces $\mathcal{A}(n,d,k,i_0,a)$ describe almost-hermitian manifolds at the K{\"a}hler threshold. We expect, for instance, that certain compact acK manifolds of real dimension $4$ 
(or higher) satisfy a Bogomolov-Miyaoka-Yau-type inequality. 

Bogomolov's pioneering work, both in the case of non-Kähler complex manifolds \cite{BogomolovVII} and of course on his inequalities of Chern numbers \cite{BogomolovIneq}, have been an inspiration for the whole theme of the study of geometrical and topological invariants of complex projective varieties and Kähler manifolds. These have been main motivations for our work.

Although our main theorem is of a qualitative nature, the quantitative aspects of this new theory will be the topic of a future article. We have included a final section on questions and future directions, which range from topology and geometry at the K{\"a}hler threshold to the automation of extension procedures in these areas with proof assitants.

The organization of the article is as follows. In section \ref{1}, we review the Hodge theory of K{\"a}hler manifolds and compactness results of Riemannian geometry that we will need in the proof of the main theorem, Theorem \ref{Betti_numbers_thm}. In section \ref{2}, we properly introduce acK manifolds. Section \ref{3} is dedicated to the proof of Theorem \ref{prepctness_thm}. Sections \ref{5} and \ref{6} deal with the analysis technical steps that go into the proof of Theorem \ref{Betti_numbers_thm}. Section \ref{7} is reserved for the proof of this theorem. The final section is on the future of the present work.

\section{Background}\label{1}

\subsection{Hodge theory}\label{1.1}

Here we outline a proof of the well-known fact that odd Betti numbers of a compact K{\"a}hler manifold are even. The proof does not pass through Dolbeault cohomology, which has not yet been studied in our not fully regular setting. Also, note that although there is an almost-complex version of Dolbeault cohomology, it does not determine a Hodge-like decomposition in general \cite{CO}.

One of the objectives of this subsection is to present the proof in a calculatory way that can then be extended to the less regular situation to be encountered later. This is why we need to write the calculations in detail. Also, this presentation could serve as a roadmap for computer formalization. 

Let $(X,J)$ be an almost-complex manifold. Let \[\Omega^{\bullet}(X)=\bigoplus_{k \geq 0} \Omega^k(X)\] be the exterior algebra of differential forms on $X.$ 

\begin{definition}\label{induced_J}
We have an endomorphism $\mathbf{J} \colon \Omega^{\bullet}(X) \to \Omega^{\bullet}(X),$ given as follows : for any $\alpha \in \Omega^k(X)$ and vector fields $\zeta_1,\dots,\zeta_k \in \mathfrak{X}(X)$, we have that \[(\mathbf{J}\alpha)(\zeta_1,\dots,\zeta_k) \coloneqq \sum_{1\leq j\leq k} \alpha(\zeta_1,\dots,J\zeta_j,\dots,\zeta_k).\] 
\end{definition}

Indeed, $\mathbf{J}$ extends the action of $J$ to $\Omega^{\bullet}(X).$ There is a completely analogous definition of $\mathbf{J}$ acting on complex differential forms. In order to simplify notation, we do not explicitly distinguish between real and complex differential forms, vector fields, and so forth. The context provided will make clear what is meant. 

Recall that for each $k$, $J$ uniquely determines a decomposition \[\Omega^k (X)=\bigoplus_{p+q=k} \Omega^{p,q}(X).\] Let $\Pi^{p,q} \colon \Omega^{\bullet}(X) \to \Omega^{p,q}(X)$ be the obvious projection. We denote $(p,q)$-forms with a superscript, i.e.\ $\alpha^{p,q}$ is such a form.

\begin{lemma}\label{poly_J}
We have that $\Pi^{p,q} \in \mathbb{C}[\mathbf{J}].$
\end{lemma}

\begin{proof}
Let $\alpha = \sum _{p+q=k}\alpha ^{p,q} \in \Omega^k (X),$ where
$\alpha (\zeta_1,\dots,\zeta_k) =\alpha^{p,q} (\zeta_1,\dots,\zeta_k)$ for any sequence of vectors fields containing $p$ holomorphic ones and $q$ antiholomorphic ones. 

Note that on a sequence of vectors fields with $\zeta_{i_1},\dots,\zeta_{i_p}$ holomorphic and 
$\zeta_{j_1},\dots,\zeta_{j_q}$ antiholomorphic, we have that

\[(\mathbf{J} \alpha ) (\zeta_1,\dots,\zeta_k) = i(p-q) \alpha^{p,q} (\zeta_1,\dots,\zeta_k)\] as $J$ acts by $\pm i$ on the $\zeta_{i_l}$'s, resp.\ the $\zeta_{j_m}$'s.

Each term $\alpha^{p,q}$ vanishes on sequences of vectors of Hodge type different from $(p,q)$. So for any sequence of vectors each of which is either holomorphic or antiholomorphic, we get
\[
(\mathbf{J}\alpha ) (\zeta_1,\dots,\zeta_k)=
\sum _{p+q=k} i(p-q) \alpha^{p,q}  (\zeta_1,\dots,\zeta_k).
\]
By linearity this holds for any sequence of vectors, i.e.
\[
(\mathbf{J}\alpha ) =
\sum _{p+q=k} i(p-q) \alpha^{p,q}  .
\]

Define the polynomial 
\[
P^{p,q} (T) \coloneqq
\frac{
\prod _{(r,s)\neq (p,q)} (T - i(r-s))
}{
\prod _{(r,s)\neq (p,q)} (i(p-q) - i(r-s))
}
\]
where the products are taken over $r+s=k$. 

We should note that for $(r,s)\neq (p,q)$ with $r+s = p+q = k$ then 
$i(p-q) - i(r-s) \neq 0$, so the denominator is nonzero. 

We have $P^{p,q}(i(r-s)) = 1$ if $(r,s)=(p,q)$ and
$P^{p,q}(i(r-s)) = 0$ if $(r,s)\neq(p,q)$.

We get
\[
P^{p,q}(\mathbf{J}) (\alpha )= 
\sum _{r+s=k}
P^{p,q}(i(r-s))\alpha ^{r,s}
= \alpha ^{p,q}.
\]
Thus, 
\[
P^{p,q}(\mathbf{J}) = \Pi ^{p,q}
\]
is the projection onto forms of Hodge type $(p,q)$. 
\end{proof}

For the rest of this section, $X$ is $n$-dimensional ($n$ will be even). Let $g$ be a Riemannian metric on $X$ with Levi-Civita connection $\nabla$, and $\star$ be the Hodge-star operator. Let $J$ be a compatible almost-complex structure, meaning that $J$ is orthogonal for $g$ and $J^2=-\operatorname{Id}_{T_X}$, in particular $J$ is also anti-self-adjoint. 
For simplicity, we work with real-valued forms.  Note that the formal adjoint of the differential is $d^*=-\star d\star.$ 

The Levi-Civita connection acts on all bundles associated to the tangent bundle, compatibly with the metrics. (We don't discuss the normalization of the metric on $\bigwedge^kT^{\ast}X$ up to a rational constant.) 

Suppose in general $V$ is a real vector bundle with metric $G$ and compatible connection $\nabla _V$. In practice typically $V=\bigwedge ^kT^{\ast}X$ and $\nabla _V$ is the Levi-Civita connection. The metric is viewed as a bilinear form  $G:V\otimes V \rightarrow \mathbb{R}$.

The Levi-Civita connection and the connection of $V$ induce a connection on $T^{\ast}X\otimes V$ which is an operator
\[
\nabla _{T^{\ast}_X \otimes V} \coloneqq \nabla _{T^{\ast}_X} \otimes \operatorname{Id}_V + \operatorname{Id}_{T^{\ast}_X} \otimes \nabla _V \colon T^{\ast}_X \otimes V \rightarrow  T^{\ast}_X\otimes T^{\ast}_X \otimes V
\]
where the differential form value of the connection is the first in the right hand tensor product. On the other hand $g$ gives the bilinear form 
\[
tr_g  \colon T^{\ast}_X \otimes T^{\ast}_X \rightarrow \mathbb{R}
\]
hence we get a map $tr_g\otimes \operatorname{Id}_V \colon T^{\ast}_X \otimes T^{\ast}_X \otimes V \rightarrow V$. 
We claim that the formal adjoint of $\nabla _V$ is the 
operator 
\[
\nabla ^{\ast} _V \coloneqq - (tr_g\otimes \operatorname{Id}_V) \circ \nabla _{T^{\ast}_X \otimes V} .
\]
To see this, note that for sections $v$ of $V$ and $A$ of $T^{\ast}X\otimes V$ we
may define $G(v,A)$ as a section of $T^{\ast}X $, and since $G$ is parallel
we get 
\[
\nabla _{T^{\ast}X} G(v,A) = 
G(\nabla _V(v), A) + G(v,\nabla _{T^{\ast}X\otimes V}(A)).
\]
If $\alpha$ is a compactly supported section of $T^{\ast}X$ we have 
\[
\int _X tr_g \nabla _{T^{\ast}X}(\alpha ) dVol  = 0
\]
by the divergence theorem, as $tr_g \nabla _{T^{\ast}X}(\alpha ) dVol$ is the divergence of the metric dual vector field $\alpha^{\sharp}$.

Now $tr_gG(\nabla _V(v), A) = (g\otimes G)(\nabla _V(v),A)$
so this gives
\[
\int _X (g\otimes G)(\nabla _V v, A) dVol - 
\int _X G(v,\nabla ^{\ast}_V A)dVol 
\]
\[
= 
\int _X tr_g \nabla _{T^{\ast}X} (G(v,A))dVol = 0. 
\]
This shows that the operator $\nabla ^{\ast}_V$ is the formal adjoint of $\nabla _V$,
justifying the notation. If $\{ e_i\}$ is a unitary frame for $TX$ then 
\[
\nabla _V^{\ast}(A) = -\sum_i \iota _{e_i} \iota _{e_i} (\nabla _W A) 
\]
where $\iota _{e_i} : T^{\ast}X\otimes W \rightarrow W$ is contraction with $e_i$, here applied to the bundle $W=T^{\ast}X \otimes V$. 

This discussion will be applied with $V= \bigwedge^k T^*_X$. 
We will drop the subscripts for the Levi-Civita covariant derivative operators and their adjoints, and notate them simply by $\nabla$, resp.\ $\nabla ^{\ast}$. 

Let $\zeta_0,\zeta_1,\dots,\zeta_k, \eta \in \mathfrak{X}(X)$ be vector fields on $X$, $\alpha \in \Omega^k(X)$, and $A\in \Omega^0 \big(X, T^*_X \otimes \bigwedge^k T^*_X \big)$. Recall that on differential forms, the covariant derivative acts in the following way :

\begin{equation*}
    \begin{split}
        (\nabla\alpha)(\eta,\zeta_1,\dots,\zeta_k)&\coloneqq (\nabla_{\eta} \alpha)(\zeta_1,\dots,\zeta_k)\\
        &=\nabla_{\eta} \big(\alpha(\zeta_1,\dots,\zeta_k)\big)-\sum_{1\leq j\leq k} \alpha(\zeta_1,\dots,\nabla_{\eta} \zeta_j,\dots,\zeta_k).
    \end{split}
\end{equation*}
 
The formal adjoint $\nabla^* \colon \Omega^0 \big(X, T^*_X \otimes \bigwedge^k T^*_X \big) \to \Omega^k(X)$ is defined as
above by
\[
\nabla^* A=-\operatorname{tr}_g (\nabla A).
\] 
Let $(e_j)_{j=1}^n$ be a local orthonormal frame of $T_X.$ Then, we can write 
\[
(\nabla^* A)(\zeta_1,\dots,\zeta_k)=-\sum_{1\leq j \leq n} (\nabla_{e_j} A)(e_j,\zeta_1,\dots,\zeta_k).
\]

The rough Laplacian is therefore given as 

\begin{equation*}
    \begin{split}
(\nabla^* \nabla \alpha) (\zeta_1,\dots,\zeta_k)&=-\sum_{1\leq j \leq n} \big(\nabla_{e_j} (\nabla \alpha)\big)(e_j,\zeta_1,\dots,\zeta_k)\\
&=-\sum_{1\leq j \leq n} \big[\nabla_{e_j} ((\nabla \alpha)(e_j,\zeta_1,\dots,\zeta_k))\\
&-(\nabla \alpha)(\nabla_{e_j} e_j,\zeta_1,\dots,\zeta_k)\\
&-\sum_{1\leq i\leq k} (\nabla_{e_j} \alpha)(\zeta_1,\dots,\nabla_{e_j} \zeta_i,\dots,\zeta_k)\big]\\
&=-\sum_{1\leq j \leq n}\big[(\nabla_{e_j} \nabla_{e_j} \alpha)(\zeta_1,\dots,\zeta_k)+\\
&\sum_{1\leq i\leq k} (\nabla_{e_j} \alpha)(\zeta_1,\dots,\nabla_{e_j} \zeta_i,\dots,\zeta_k)\big]+\\
&\sum_{1\leq j \leq n}\big[(\nabla_{\nabla_{e_j} e_j} \alpha)(\zeta_1,\dots,\zeta_k)+\\
&\sum_{1\leq i\leq k} (\nabla_{e_j} \alpha)(\zeta_1,\dots,\nabla_{e_j} \zeta_i,\dots,\zeta_k)\big]\\
&=-\sum_{1\leq j \leq n}\big[(\nabla_{e_j} \nabla_{e_j} \alpha)(\zeta_1,\dots,\zeta_k)-(\nabla_{\nabla_{e_j} e_j} \alpha)(\zeta_1,\dots,\zeta_k)\big].
\end{split}
\end{equation*}

More succinctly, we can write \[\nabla^* \nabla \alpha = -\sum_{1\leq j \leq n}\big[(\nabla_{e_j} \nabla_{e_j} \alpha)-(\nabla_{\nabla_{e_j} e_j} \alpha)\big].\]

In addition, we have that 

\[d{\alpha}(\zeta_0,\dots,\zeta_k)=\sum_{0\leq j\leq k} (-1)^j (\nabla_{\zeta_j} \alpha)(\zeta_0,\dots,\hat{\zeta}_j,\dots,\zeta_k)\] and 

\[d^* {\alpha}(\zeta_1,\dots,\zeta_{k-1})=-\sum_{1\leq j\leq n}(\nabla_{e_j} \alpha)(e_j,\zeta_1,\dots,\zeta_{k-1}).\]

The extension of the curvature endomorphism of $\nabla$ to differential forms is given as $R^{\nabla}\alpha \coloneqq (d^{\nabla})^2 \alpha,$ where \[(R^{\nabla} \alpha)(\eta,\zeta_0) \coloneq R^{\nabla} (\eta,\zeta_0)\alpha =([\nabla_{\eta},\nabla_{\zeta_0}]-\nabla_{[\eta,\zeta_0]})\alpha.\] The Weitzenb{\"o}ck curvature operator is defined as \[(R\alpha)(\zeta_1,\dots,\zeta_k)=\sum_{\substack{1\leq i\leq k \\ 1\leq j \leq n}} (R^{\nabla}(e_j,\zeta_i)\alpha)(\zeta_1,\dots,e_j,\dots,\zeta_k),\] where indeed in each term of the latter sum $e_j$ is located in place $i$.

We are going to derive and use the following version of the Weitzenb{\"o}ck formula for differential forms to prove the Hodge decomposition of complex De Rham cohomology into $(p,q)$-harmonic forms. We usually remain with real coefficients in order to avoid notations with hermitian structures and complex conjugation. The passage to complex coefficients is only needed when defining the $(p,q)$ components as eigenspaces of $\mathbf{J}$.

Recall the $d$-Laplacian $\Delta_d \colon \Omega^k(X) \to \Omega^k(X),$ where $\Delta_d=dd^* + d^* d.$ For the formally identical Weitzenb{\"o}ck formula, see \cite{Petersen_book} (Theorem 50).

\begin{proposition}{(Weitzenb{\"o}ck formula)}\label{wf_thm}
For any $\alpha \in \Omega^k(X)$, \[\Delta_d \alpha=\nabla^* \nabla \alpha + R\alpha,\] where $R$ is the Weitzenb{\"o}ck curvature operator.
\end{proposition}

\begin{proof}
Since both sides of the equation we want to prove are tensorial, it suffices to verify it at an arbitrary point $p\in X.$ Assume that $p$ is the center of a normal coordinate chart with associated normal coordinate frame $(e_j)_{j=1}^n$ of $T_X.$ Observe that at $p$, we have that

\begin{equation*}
    \begin{split}
        dd^* \alpha(e_1,\dots,e_k)&=\sum_{1\leq i\leq k} (-1)^{i+1} (\nabla_{e_i} d^* \alpha)(e_1,\dots,\hat{e}_i,\dots,e_k)\\
        &=\sum_{\substack{1\leq i\leq k \\1\leq j\leq n}} (-1)^i (\nabla_{e_i} \nabla_{e_j} \alpha) (e_j,e_1,\dots,\hat{e}_i,\dots,e_k)\\
        &=-\sum_{\substack{1\leq i\leq k \\1\leq j\leq n}} (\nabla_{e_i} \nabla_{e_j} \alpha) (e_1,\dots,e_{i-1},e_j,e_{i+1},\dots,e_k).
    \end{split}
\end{equation*}

We also have that at $p$,

\begin{equation*}
    \begin{split}
        d^* d \alpha(e_1,\dots,e_k)&=-\sum_{1\leq j\leq n} (\nabla_{e_j}  d \alpha)(e_j,e_1,\dots,e_k)\\
        &=-\sum_{1\leq j\leq n} (\nabla_{e_j} \nabla_{e_j} \alpha)(e_1,\dots,e_k)\\
        &-\sum_{\substack{1\leq i\leq k \\1\leq j\leq n}} (-1)^i (\nabla_{e_j} \nabla_{e_i} \alpha)(e_j,e_1,\dots,\hat{e}_i,\dots,e_k)\\
        &=(\nabla^* \nabla \alpha)(e_1,\dots,e_k)+\\
        &\sum_{\substack{1\leq i\leq k \\1\leq j\leq n}} (\nabla_{e_j} \nabla_{e_i} \alpha)(e_1,\dots,e_{i-1},e_j,e_{i+1},\dots,e_k).
    \end{split}
\end{equation*}

Hence, 

\begin{equation*}
    \begin{split}
        (\Delta_d \alpha)(e_1,\dots,e_k)&=(\nabla^* \nabla \alpha)(e_1,\dots,e_k)+\\
        &\sum_{\substack{1\leq i\leq k \\1\leq j\leq n}} \big((\nabla_{e_j} \nabla_{e_i}-\nabla_{e_i} \nabla_{e_j})\alpha\big) (e_1,\dots,e_{i-1},e_j,e_{i+1},\dots,e_k)\\
        &=(\nabla^* \nabla \alpha)(e_1,\dots,e_k)+\\
        &\sum_{\substack{1\leq i\leq k \\1\leq j\leq n}} 
        \big(R^{\nabla}(e_j,e_i)\alpha\big) (e_1,\dots,e_{i-1},e_j,e_{i+1},\dots,e_k)\\
        &=\big(\nabla^* \nabla \alpha +R \alpha \big)(e_1,\dots,e_k).
    \end{split}
\end{equation*}
\end{proof}

The graded commutator of a pair of graded linear operators $L_1,L_2$ on $\Omega^{\bullet}(X)$ of respective degrees $l_1,l_2$ is given as \[[L_1,L_2] \coloneqq L_1 L_2 - (-1)^{l_1 l_2} L_2 L_1.\] The $d$-Laplacian and $\mathbf{J}$ are graded linear operators of degree $0$, so $[\Delta_d,\mathbf{J}]=\Delta_d \mathbf{J} - \mathbf{J} \Delta_d.$ 

From this point on, we assume that $(X,g,J)$ is a compact K{\"a}hler manifold with K{\"a}hler form $\omega_J.$

\begin{proposition}\label{DeltaJ}
The $d$-Laplacian commutes with $\mathbf{J}$, i.e.\ \[[\Delta_d,\mathbf{J}]=0.\]
\end{proposition}

\begin{proof}
We show that $[\nabla^* \nabla,\mathbf{J}]=0$ and that $[R,\mathbf{J}]=0,$ c.f.\ Proposition \ref{wf_thm}. Let $(e_j)_{j=1}^n$ be a local orthonormal frame of $T_X$, and $\alpha \in \Omega^k(X).$ Using the Leibniz rule and the K{\"a}hler condition $\nabla J=0,$ which implies that $\nabla \mathbf{J}=0,$ we compute

\begin{equation*}
    \begin{split}
        (\nabla^* \nabla \mathbf{J})\alpha &=\nabla^* \nabla (\mathbf{J}\alpha)\\
        &=-\sum_{1\leq j\leq n} \big(\nabla_{e_j} \nabla_{e_j} (\mathbf{J}\alpha)-\nabla_{\nabla_{e_j} e_j} (\mathbf{J}\alpha)\big)\\
        &=-\sum_{1\leq j\leq n} \big(\nabla_{e_j} (\mathbf{J} \nabla_{e_j} \alpha)-\mathbf{J}\nabla_{\nabla_{e_j} e_j} \alpha\big)\\
        &=-\sum_{1\leq j\leq n} \big(\mathbf{J}\nabla_{e_j} \nabla_{e_j} \alpha-\mathbf{J}\nabla_{\nabla_{e_j} e_j} \alpha\big)\\
        &=(\mathbf{J} \nabla^* \nabla)\alpha,
    \end{split}
\end{equation*}
i.e.\ $[\nabla^* \nabla,\mathbf{J}]=0.$

Now, we tackle the proof of $[R,\mathbf{J}]=0$. For the rest of the proof, $\nabla$ is the induced connection on $\bigwedge^k T^*_X$ unless otherwise stated. Moreover, we view $\alpha$ as belonging to the space of $0$-forms with values in $\bigwedge^k T^*_X,$ notated $\alpha \in \Omega^0 \Big(X,\bigwedge^k T^*_X \Big).$ In this way, we have the identity $(d^{\nabla})^2 \alpha = R^{\nabla} \alpha \in  \Omega^2 \Big(X,\bigwedge^k T^*_X \Big).$ Indeed, here $R^{\nabla} \in \Omega^2 \Big(X,\operatorname{End}\big(\bigwedge^k T^*_X\big) \Big).$ Since $d^{\nabla} J=0$ because $J$ is K{\"a}hler, where this time $\nabla$ is the original Levi-Civita connection, it follows that $d^{\nabla} \mathbf{J} =0$. By the Leibniz rule, \[R^{\nabla} (\mathbf{J}\alpha)=(d^{\nabla})^2 (\mathbf{J}\alpha) =\mathbf{J} \circ (R^{\nabla} \alpha),\] where for any $\zeta, \eta \in \mathfrak{X}(X),$ $R^{\nabla} (\zeta,\eta) \in \operatorname{End}\big(\bigwedge^k T^*_X\big),$ so in $\mathbf{J} \circ (R^{\nabla} \alpha)$, we should think of $\mathbf{J}$ as if acting on the $k$-form $R^{\nabla} (\zeta,\eta) \alpha.$ From these observations, it follows that

\begin{equation*}
    \begin{split}
        \big(R (\mathbf{J}\alpha)\big)(\zeta_1,\dots,\zeta_k) &=\sum_{\substack{1\leq i \leq k \\ 1\leq j \leq n}} \big(R^{\nabla}(e_j,\zeta_i) (\mathbf{J}\alpha)\big) (\zeta_1,\dots,e_j,\dots,\zeta_k)\\
        &=\sum_{\substack{1\leq i \leq k \\ 1\leq j \leq n}} \Big(\mathbf{J} \big(R^{\nabla}(e_j,\zeta_i) \alpha\big)\Big)(\zeta_1,\dots,e_j,\dots,\zeta_k)\\
        &=\sum_{\substack{1\leq i \leq k \\ 1\leq j \leq n}} \Big[\sum_{l<i} \big(R^{\nabla}(e_j,\zeta_i) \alpha\big)(\zeta_1,\dots,J\zeta_l,\dots,e_j,\dots,\zeta_k)+\\
        &\big(R^{\nabla}(e_j,\zeta_i) \alpha\big)(\zeta_1,\dots,Je_j,\dots,\zeta_k)+\\
        &\sum_{l>i} \big(R^{\nabla}(e_j,\zeta_i) \alpha\big)(\zeta_1,\dots,e_j,\dots,J\zeta_l,\dots,\zeta_k)\Big].
    \end{split}
\end{equation*}

Next, we note that 

\begin{equation*}
    \begin{split}
        \big(\mathbf{J}(R\alpha)\big)(\zeta_1,\dots,\zeta_k) &=\sum_{1\leq l \leq k} (R\alpha) (\zeta_1,\dots,J\zeta_l,\dots,\zeta_k)\\
        &=\sum_{\substack{1\leq l \leq k \\ 1\leq j \leq n}} \Big[\sum_{i<l} \big(R^{\nabla}(e_j,\zeta_i) \alpha\big)(\zeta_1,\dots,e_j,\dots,J\zeta_l,\dots,\zeta_k)+\\
        &\big(R^{\nabla}(e_j,J\zeta_l) \alpha\big)(\zeta_1,\dots,e_j,\dots,\zeta_k)+\\
        &\sum_{i>l} \big(R^{\nabla}(e_j,\zeta_i) \alpha\big)(\zeta_1,\dots,J\zeta_l,\dots,e_j,\dots,\zeta_k)\Big].
    \end{split}
\end{equation*}

Based on the previous $2$ computations, we know that 

\begin{equation*}
    \begin{split}
        \big(R(\mathbf{J} \alpha)-\mathbf{J}(R\alpha)\big)(\zeta_1,\dots,\zeta_k) &=\sum_{\substack{1\leq i \leq k\\ 1\leq j \leq n}} \Big[\big(R^{\nabla}(e_j,\zeta_i) \alpha\big)(\zeta_1,\dots,Je_j,\dots,\zeta_k)\\
        &-\big(R^{\nabla}(e_j,J\zeta_i) \alpha\big)(\zeta_1,\dots,e_j,\dots,\zeta_k)\Big].
    \end{split}
\end{equation*}

Using (16) on page 32 of \cite{Moroianu}, we see that the Riemann curvature tensor, $Rm$, satisfies \[Rm(e_j,J\zeta_i,\rho,\mu)=-Rm(Je_j,\zeta_i,\rho,\mu)\] for all vector fields $\rho, \mu$, and all $i,j$. But the latter can happen iff $-R^{\nabla}(e_j,J\zeta_i)=R^{\nabla}(Je_j,\zeta_i).$ Notice that $(Je_j)_{j=1}^n$ is also a local orthonormal frame of $T_X.$ Hence, the right hand side of the expression

\[-\big(R^{\nabla}(e_j,J\zeta_i) \alpha\big)(\zeta_1,\dots,e_j,\dots,\zeta_k)=\big(R^{\nabla}(Je_j,\zeta_i) \alpha\big)(\zeta_1,\dots,e_j,\dots,\zeta_k)\] can be rewritten relative to $(Je_j)_{j=1}^n$ in the following way : 

\[-\big(R^{\nabla}(e_j,\zeta_i) \alpha\big)(\zeta_1,\dots,Je_j,\dots,\zeta_k).\] Immediately, we obtain that 

\begin{equation*}
    \begin{split}
        \big(R(\mathbf{J} \alpha)-\mathbf{J}(R\alpha)\big)(\zeta_1,\dots,\zeta_k) &=\sum_{\substack{1\leq i \leq k\\ 1\leq j \leq n}} \Big[\big(R^{\nabla}(e_j,\zeta_i) \alpha\big)(\zeta_1,\dots,Je_j,\dots,\zeta_k)\\
        &-\big(R^{\nabla}(e_j,J\zeta_i) \alpha\big)(\zeta_1,\dots,e_j,\dots,\zeta_k)\Big]\\
        &=0.
    \end{split}
\end{equation*}

\end{proof}

\begin{corollary}\label{DeltaJm}
For any $m\in \mathbb{N}$, let $\mathbf{J}^m$ be the $m$-time self-composition of $\mathbf{J}$. Then, \[[\Delta_d,\mathbf{J}^m]=0.\]
\end{corollary}

\begin{proof}
We prove the statement by induction. 
The base step of the induction is covered by Proposition \ref{DeltaJ}. Assume that the claim holds true for a given $m.$ Then, 
\begin{equation*}
    \begin{split}
        [\Delta_d,\mathbf{J}^{m+1}]&=\Delta_d\mathbf{J}^{m+1} - \mathbf{J}^{m+1} \Delta_d\\
        &=\Delta_d\mathbf{J}^m \mathbf{J}-\mathbf{J}^m \mathbf{J} \Delta_d\\
        &=\mathbf{J}^m \Delta_d \mathbf{J}-\mathbf{J}^m \mathbf{J} \Delta_d\\
        &=\mathbf{J}^m [\Delta_d,\mathbf{J}]\\
        &=0.
    \end{split}
\end{equation*}
\end{proof}

The following is now an immediate consequnce of the bilinearity of the commutator and Corollary \ref{DeltaJm}.

\begin{corollary}\label{DeltaJP}
For any polynomial $\mathbf{p} \in \mathbb{C}[\mathbf{J}]$, \[[\Delta_d,\mathbf{p}]=0.\]
\end{corollary}

\begin{theorem}\label{Hodgethm1}
Let 
\[
\mathcal{H}^k(X) \coloneqq \ker{(\Delta_d)}
\] 
be the vector space of $d$-harmonic complex $k$-forms, and 
\[
\mathcal{H}^{p,q}(X) \coloneqq \mathcal{H}^k(X)
\otimes_{\mathbb{R}}\! \mathbb{C} \cap \Omega^{p,q}(X).
\] 
Then, 
\[
\mathcal{H}^k(X)\otimes_{\mathbb{R}}\! \mathbb{C}=\bigoplus_{p+q=k} \mathcal{H}^{p,q}(X).
\] 
\end{theorem}

\begin{proof}
We explain the containment 
$\mathcal{H}^k(X)\otimes_{\mathbb{R}}\!\mathbb{C} \subseteq \bigoplus_{p+q=k} \mathcal{H}^{p,q}(X).$ The reverse containment is basic linear algebra: as $\mathcal{H}^{p,q}(X) \subseteq \mathcal{H}^k(X) \otimes_{\mathbb{R}}\!\mathbb{C},$ it follows that $\bigoplus_{p+q=k} \mathcal{H}^{p,q}(X) \subseteq \mathcal{H}^k(X)\otimes_{\mathbb{R}}\!\mathbb{C}$. Let $\alpha=\sum_{p+q=k} \alpha^{p,q} \in \mathcal{H}^k(X)
\otimes_{\mathbb{R}}\!\mathbb{C}$. By Lemma \ref{poly_J} and Corollary \ref{DeltaJP}, we have that $[\Delta_d,\Pi^{p,q}]=0.$ So, \[\Delta_d \alpha^{p,q}=\Delta_d \Pi^{p,q} \alpha=\Pi^{p,q} \Delta_d \alpha=0,\] and thus $\alpha^{p,q} \in \mathcal{H}^{p,q}(X)$ so that $\alpha \in \bigoplus_{p+q=k} \mathcal{H}^{p,q}(X).$ 
\end{proof}

Recall the following important theorem of Riemannian geometry. The proof of Theorem \ref{Hodgethm2} can be found in various books on differential and complex geometry. 

\begin{theorem}{(Hodge isomorphism theorem, Proposition 9.3 \cite{Moroianu})}\label{Hodgethm2}
There is an isomorphism of vector spaces \[H^k_{DR}(X) \xrightarrow{\simeq} \mathcal{H}^k(X).\]
\end{theorem}

\begin{theorem}\label{Betti}
For any $l \geq 0$, the Betti numbers satisfy $b_{2l+1}(X) \in 2\mathbb{Z}_{\geq 0}.$
\end{theorem}

\begin{proof}
Note that $\beta \in \Omega^{p,q}(X)$ iff $\overline{\beta} \in \Omega^{q,p}(X),$ i.e.\ $\overline{\Omega^{p,q}(X)}=\Omega^{q,p}(X).$ Since the $d$-Laplacian is a real operator, it is invariant under conjugation, so \[\overline{\mathcal{H}^{p,q}(X)}=\{\beta \in \Omega^{q,p}(X) \mid \Delta_d \beta=0 \}=\mathcal{H}^{q,p}(X).\]  It then follows from Theorems \ref{Hodgethm1} and \ref{Hodgethm2} that

\begin{equation*}
    \begin{split}
        b_{2l+1}(X)& \coloneqq \dim{H^{2l+1}_{DR} (X) \otimes_{\mathbb{R}}\!\mathbb{C}}\\
        &=\sum_{p+q=2l+1} \dim{\mathcal{H}^{p,q}(X)}\\
        &=\sum_{\substack{p+q=2l+1 \\ p<q}} \dim{\mathcal{H}^{p,q}(X)}+ \sum_{\substack{p+q=2l+1 \\ q<p}} \dim{\mathcal{H}^{p,q}(X)}\\
        &=2 \sum_{\substack{p+q=2l+1 \\ p<q}} \dim{\mathcal{H}^{p,q}(X)}\\
        &\in 2\mathbb{Z}_{\geq 0}.
    \end{split}
\end{equation*}
\end{proof}

\begin{remark}\label{evenBetti}
For any $k$, $b_{2k}(X)>0$. To see why this is the case, note that $\nabla \omega_J =0.$ A quick induction argument confirms that $\nabla \omega^k_J =0.$ Since $d$ and $d^*$ are defined in terms of $\nabla,$ it is clear that $d \omega^k_J =0$ and $d^* \omega^k_J =0.$ So $\Delta_d \omega^k_J =0.$ We know that $\omega^k_J$ must represent a De Rham cohomology class thanks to the Hodge isomorphism theorem. Indeed, since $g$ is nondegenerate and $\omega _J$ is pointwise the Kähler form of the complex structure, we have $\omega_J^k \neq 0$ for $k\leq n/2$. Thus, $H^{2k}_{DR}(X)$ is non-trivial.
\end{remark}

\subsection{Compactness theorems}\label{1.2}
In this subsection, we present Anderson's compactness theorem \cite{Anderson,Anderson_notes}, which gives an optimal statement for regularity. See Theorem \ref{Acpctness}. We also mention some compactness results that appeared earlier in the literature.

Let $k\geq 0$, $0<p\leq \infty$ be integers, and $0<\alpha < 1$ be a real number. The most relevant Banach spaces here are the Sobolev spaces $W^{k,p}$ of functions with weak derivatives up to order $k$ belonging in $L^p$, and the H{\"o}lder spaces $C^{k,\alpha}$ of $C^k$ functions with partial derivatives of order $k$ that are $\alpha$-H{\"o}lder continuous. Recall too that any Banach space induces a topology; i.e.\ the norm $\|\cdot\|$ defines the norm topology on $B$ through the metric $d(g,h) \coloneqq \|g - h\|.$

\begin{definition}\label{metric_regularity}
Let $X$ be a smooth manifold, $B$ be a Banach space of real-valued functions on $X,$ and $g$ be a metric on $X$ (i.e. a section of the vector bundle $T^{*}_X \otimes T^{*}_X$ that is symmetric and positive definite). We say that $g$ is a $B$ metric if its coefficient functions belong to $B.$ If $g$ is a $B$ metric, then we call $(X,g)$ a $B$ Riemannian manifold.
\end{definition}

For example, as per Definition \ref{metric_regularity}, a Riemannian metric is a $C^{\infty}$ metric, and a Riemannian manifold is a $C^{\infty}$ Riemannian manifold. $B$ tensor fields are defined similarly. For instance, $B$ almost-complex structures have coefficients in $B$.

\begin{definition}\label{metric_convergence}
Let $B$ be a Banach space of real valued functions defined on a smooth manifold $X$ of dimension $n.$ Let $(X,g)$ be a $B$ Riemannian manifold. A sequence $\{(X_i,g_i)\}$ of $n$-dimensional Riemannian manifolds converges (resp. converges weakly) in the $B$ topology to $(X,g)$ if for sufficiently large $i,$ there is a sequence of diffeomorphisms $\{f_i : X\to X_i\}$ such that $\{f^*_i g_i\}$ converges (resp. converges weakly) to $g$ in the $B$ topology on $X$.
\end{definition}

In Definition \ref{metric_convergence}, the precise meaning of the sequence $\{f^*_i g_i\}$ converging to $g$ in the $B$ topology is that there is a finite collection of charts $\{\phi_k\}$ covering $X$ such that the sequences of metric component functions relative to the $\phi_k,$ $\{(f^*_i g_i)_{ab}\},$ converge to the metric components $g_{ab}$ in the $B$ topology; i.e.\ w.r.t.\ the $\phi_k,$ for all $a,b,$ we have that $\|(f^*_i g_i)_{ab} - g_{ab}\|\to 0$ as $i\to \infty.$

Note that in our setting, sequential compactness implies precompactness.

Let $n \geq 1$ be an integer, $d,l,v>0$ be real numbers, and $\mathcal{M}(n,d,l,v)$ be the class of compact $n$-dimensional Riemannian manifolds with diameter, sectional curvature, and volume respectively bounded as $\diam \leq d,$ $|sec|\leq l^2$ and $\vol\geq v$. The Cheeger-Gromov compactness theorem is the statement that $\mathcal{M}(n,d,l,v)$ is precompact in the $C^{1,\alpha}$ topology (see Theorem 4.4 \cite{SPetersen}, and the original references cited therein). The proof is a sequential compactness argument that makes use of harmonic coordinates to obtain optimal regularity of the limit metric. For the definition of these coordinates, see Lemma 48 \cite{Petersen_book}. 

We end this subsection with the following generalization of Cheeger-Gromov compactness due to Anderson, which will be key for proving our main theorem.
For this we choose some $p\gg 0$ sufficiently big so that the appropriate Sobolev embedding and multiplication properties hold. 

Note that $W^{2,p}$ is reflexive so the weak and weak-(*) topologies are the same and any bounded sequence has a subsequence converging weakly in $W^{2,p}$. We use this to rephrase the statement of Anderson's theorem in a way that will suit our needs. 

\begin{definition}\label{MI_def}
The class of compact $n$-dimensional Riemannian manifolds satisfying $\diam \leq d,$ $\|Ric\|\leq k$ and $i\geq i_0,$ where $Ric$ is the Ricci curvature and $i$ is the injectivity radius is denoted by $\mathcal{M}(n,d,k,i_0).$ 
\end{definition}

\begin{theorem}{(Theorem 2.2 \cite{Anderson_notes})}\label{Acpctness}
Any sequence in $\mathcal{M}(n,d,k,i_0)$ has a subsequence converging in the weak $W^{2,p}$ topology to a $W^{2,p}$ Riemannian manifold. The convergence is strong in $C^{1,\alpha}$ for $0<\alpha < 1$ depending on $p$.
\end{theorem}

For a small point to be used later, we comment on the unicity of the limit. 
The discussion of \cite{Anderson_notes} entails getting a lower bound for the $W^{2,p}$ harmonic radius, then using harmonic coordinates and partitions of unity to embed the manifolds into $\mathbb{R}^N$ in order to take the limit. The image of the limiting manifold has regularity $W^{3,p}$ as a subset of $\mathbb{R}^N$.  
Given two such limiting manifolds $(X_{\infty},g_{\infty})$ and 
$(X'_{\infty},g'_{\infty})$, the combined limiting image in $\mathbb{R}^{2N}$ provides the graph of a $W^{3,p}$ isomorphism $X_{\infty}\cong X'_{\infty}$ sending $g_{\infty}$ to $g'_{\infty}$. 

\section{Bounded almost-complex-K{\"a}hler manifolds}\label{2}

For the moment we are working with smooth manifolds, although it could be useful to make the definitions that follow in the weaker setting too. It is left to the reader to introduce notation for that. 

\begin{definition}\label{def_norm}
Let $(X,g)$ be an $n$-dimensional Riemannian manifold and $(e_i)_{i=1}^n$ be a local orthonormal frame of $T_X$. Let $\alpha \in \Omega^k(X)$ and $\beta \in \Omega^k(X,T_X)$. The norms of $\alpha$ and $\beta$ are defined by the formulas

\[\|\alpha\|^2_g=\frac{1}{k!} \sum_{1\leq i_1,\dots,i_k\leq n} |\alpha(e_{i_1},\dots,e_{i_k})|^2\] and \[\|\beta\|^2_g=\frac{1}{k!} \sum_{1\leq i_1,\dots,i_k\leq n} \|\beta(e_{i_1},\dots,e_{i_k})\|^2_g\] where $|\alpha(e_{i_1},\dots,e_{i_k})|$ is the absolute value of the function $\alpha(e_{i_1},\dots,e_{i_k})$ and $\|\beta(e_{i_1},\dots,e_{i_k})\|^2_g=g\big(\beta(e_{i_1},\dots,e_{i_k}),\beta(e_{i_1},\dots,e_{i_k})\big)$.
\end{definition}

\begin{lemma}\label{norm_bound}
Let $\alpha \in \Omega^k(X)$ and $\beta \in \Omega^k(X,T_X)$. For any vector fields $\zeta_1,\dots,\zeta_k \in \mathfrak{X}(X)$, we have that \[|\alpha(\zeta_1,\dots,\zeta_k)|^2 \leq k! \|\alpha\|^2_g \|\zeta_1\|^2_g \dots \|\zeta_k\|^2_g\] and similarly \[\|\beta(\zeta_1,\dots,\zeta_k)\|^2_g \leq k! \|\beta\|^2_g \|\zeta_1\|^2_g \dots \|\zeta_k\|^2_g.\]

In particular, \[|\alpha(\zeta_1,\dots,\zeta_k)| \leq \sqrt{k!} \|\alpha\|_g \|\zeta_1\|_g \dots \|\zeta_k\|_g\] and \[\|\beta(\zeta_1,\dots,\zeta_k)\|_g \leq \sqrt{k!} \|\beta\|_g \|\zeta_1\|_g \dots \|\zeta_k\|_g.\]
\end{lemma}

\begin{proof}
For the first claim, observe that relative to a local orthonormal frame, we can write, for any $1\leq j \leq k$, $\zeta_j=\sum_{1\leq i_j \leq n} \zeta^{i_j}_j e_{i_j}$ so that \[\alpha(\zeta_1,\dots,\zeta_k)=\sum_{1\leq i_1, \dots,i_k \leq n} \zeta^{i_1}_1 \dots \zeta^{i_k}_k \alpha(e_{i_1},\dots,e_{i_k}). \] For any $k$-tuple $I=(i_1, \dots,i_k)$ of natural numbers, define $\zeta^I = \zeta^{i_1}_1 \dots \zeta^{i_k}_k$ and $\alpha^I = \alpha(e_{i_1},\dots,e_{i_k})$. Then, we can apply the Cauchy-Schwarz inequality for the square of the absolute value of a sum of products of real-valued functions in the following way :

\begin{equation*}
    \begin{split}
       |\alpha(\zeta_1,\dots,\zeta_k)|^2&= \big|\sum_{1\leq i_1, \dots,i_k \leq n} \zeta^{i_1}_1 \dots \zeta^{i_k}_k \alpha(e_{i_1},\dots,e_{i_k})\big|^2\\
       &=\big|\sum_{I\in \mathbb{N}^k_{\leq n}} \zeta^I \alpha^I\big|^2\\
       &\leq \big(\sum_{I\in \mathbb{N}^k_{\leq n}} (\zeta^I)^2 \big) \big(\sum_{I\in \mathbb{N}^k_{\leq n}} (\alpha^I)^2 \big)\\
       &=\big(\sum_{1\leq i_1, \dots,i_k \leq n} |\zeta^{i_1}_1 \dots \zeta^{i_k}_k|^2 \big) \big(\sum_{1\leq i_1, \dots,i_k \leq n} |\alpha(e_{i_1},\dots,e_{i_k})|^2 \big)\\
       &=\big(\sum_{1\leq i_1\leq n} |\zeta^{i_1}_1|^2 \dots \sum_{1\leq i_k\leq n} |\zeta^{i_k}_k|^2\big)\big(\sum_{1\leq i_1, \dots,i_k \leq n} |\alpha(e_{i_1},\dots,e_{i_k})|^2 \big)\\
       &=k! \|\alpha\|^2_g \|\zeta_1\|^2_g \dots \|\zeta_k\|^2_g.
    \end{split}
\end{equation*}

We proceed in a similar fashion to prove the second inequality. Set $\beta^I = \beta(e_{i_1},\dots,e_{i_k})$. Then,

\begin{equation*}
    \begin{split}
       \|\beta(\zeta_1,\dots,\zeta_k)\|^2_g&= \big\|\sum_{1\leq i_1, \dots,i_k \leq n} \zeta^{i_1}_1 \dots \zeta^{i_k}_k \beta(e_{i_1},\dots,e_{i_k})\big\|^2_g\\
       &=\big\|\sum_{I\in \mathbb{N}^k_{\leq n}} \zeta^I \beta^I\big\|^2\\
       &=\sum_{I,I' \in \mathbb{N}^k_{\leq n}} \zeta^I \zeta^{I'} g(\beta^I,\beta^{I'}) \\
       &\leq \sum_{I,I' \in \mathbb{N}^k_{\leq n}} |\zeta^I| |\zeta^{I'}| |g(\beta^I,\beta^{I'})|.
       \end{split}
\end{equation*}

Now, the Cauchy-Schwarz inequality for the absolute value of the Riemannian metric inner product of $2$ vector fields give us that \[|g(\beta^I,\beta^{I'})|\leq \|\beta^I\|_g \|\beta^{I'}\|_g\] so

\begin{equation*}
    \begin{split}
       \|\beta(\zeta_1,\dots,\zeta_k)\|^2_g& \leq \sum_{I,I' \in \mathbb{N}^k_{\leq n}} |\zeta^I| |\zeta^{I'}| |g(\beta^I,\beta^{I'})|\\
       &\leq \sum_{I,I' \in \mathbb{N}^k_{\leq n}} |\zeta^I| |\zeta^{I'}| \|\beta^I\|_g \|\beta^{I'}\|_g \\
       &\leq \big|\sum_{I \in \mathbb{N}^k_{\leq n}} |\zeta^I| \|\beta^I\|_g \big|^2\\
       &\leq \big(\sum_{I \in \mathbb{N}^k_{\leq n}} |\zeta^I|^2 \big) \big(\sum_{I \in \mathbb{N}^k_{\leq n}} \|\beta^I\|^2_g \big)\\
       &=k! \|\beta\|^2_g \|\zeta_1\|^2_g \dots \|\zeta_k\|^2_g,
    \end{split}
\end{equation*}
where note we used the Cauchy-Schwarz inequality again to go from the $3^{rd}$ to the $4^{th}$ line.
\end{proof}

\begin{definition}\label{Nij_def}
The Nijenhuis tensor of an almost-complex structure $J$ on a smooth manifold $X$ is given by the formula
\[N_J(\zeta,\eta)=[J\zeta,J\eta]-J\big([J\zeta,\eta]+[\zeta,J\eta]\big)-[\zeta,\eta]\] for any $\zeta,\eta \in \mathfrak{X}(X)$.
\end{definition}

\begin{lemma}{(Proposition 4.2 \cite{KN})}\label{KN_formula}
Let $(X,g,J)$ be an almost-hermitian manifold with fundamental $2$-form $\omega_J \coloneqq g(J\cdot,\cdot)$, and let $\nabla$ be the Levi-Civita connection of $g$. For any $\zeta,\eta,\rho \in \mathfrak{X}(X)$, we have that 

\[2g\big((\nabla_{\zeta} J)\eta, \rho\big)=d\omega_J(\zeta,\eta,\rho)-d\omega_J(\zeta,J\eta,J\rho)+g\big(N_J(\eta,\rho), J\zeta\big)\]
\end{lemma}

\begin{lemma}\label{cov_kaehler}
Let $(X,g,J)$ be an almost-hermitian manifold with fundamental $2$-form $\omega_J$, and let $\nabla$ be the Levi-Civita connection of $g$. Then, \[(\nabla_{\zeta} \omega_J)(\eta,\rho)=g\big((\nabla_{\zeta} J)\eta,\rho\big)\]
\end{lemma}

\begin{proof}
Since \[\nabla_{\zeta} g(J\eta,\rho)=g(\nabla_{\zeta}J\eta,\rho)+g(J\eta,\nabla_{\zeta} \rho),\] we have that 

\begin{equation*}
    \begin{split}
        (\nabla_{\zeta} \omega_J)(\eta,\rho)&= \nabla_{\zeta} g(J\eta,\rho)-g(J\nabla_{\zeta} \eta,\rho)-g(J\eta,\nabla_{\zeta} \rho)\\
        &=g(\nabla_{\zeta}J\eta,\rho)+g(J\eta,\nabla_{\zeta} \rho)-g(J\nabla_{\zeta} \eta,\rho)-g(J\eta,\nabla_{\zeta} \rho)\\
        &=g\big((\nabla_{\zeta} J)\eta,\rho\big).
    \end{split}
\end{equation*}
\end{proof}

\begin{proposition}\label{KT_lemma}
Let $(X,g,J)$ be an almost-hermitian manifold with fundamental $2$-form $\omega_J$, and let $\nabla$ be the Levi-Civita connection of $g$. Then, $\|d\omega_J\|_g \leq \epsilon$ and $\|N_J\|_g \leq \epsilon'$ for some $\epsilon,\epsilon' >0$ if and only if $\|\nabla J\|_g \leq \epsilon''$ for some $\epsilon'' >0$. 
\end{proposition}

\begin{proof}
We prove the forward direction first. Assume that $\|d\omega_J\|_g \leq \epsilon$ and $\|N_J\|_g \leq \epsilon'$ for some $\epsilon,\epsilon' >0$. Let $\zeta,\eta \in \mathfrak{X}(X)$. From Lemma \ref{KN_formula} with $\rho=(\nabla_{\zeta} J)\eta$, together with the triangle and Cauchy-Schwarz inequalities, it follows that

\begin{equation*}
    \begin{split}
        \|(\nabla_{\zeta} J)\eta\|^2_g &=\frac{1}{2}\Big(d\omega_J(\zeta,\eta,(\nabla_{\zeta} J)\eta)-d\omega_J(\zeta,J\eta,J(\nabla_{\zeta} J)\eta)+g\big(N_J(\eta,(\nabla_{\zeta} J)\eta), J\zeta\big) \Big)\\
        &=\frac{1}{2}\Big|d\omega_J(\zeta,\eta,(\nabla_{\zeta} J)\eta)-d\omega_J(\zeta,J\eta,J(\nabla_{\zeta} J)\eta)+g\big(N_J(\eta,(\nabla_{\zeta} J)\eta), J\zeta\big) \Big|\\
        &\leq \frac{1}{2}\Big|d\omega_J(\zeta,\eta,(\nabla_{\zeta} J)\eta)\Big|+\frac{1}{2}\Big|d\omega_J(\zeta,J\eta,J(\nabla_{\zeta} J)\eta)\Big|+\\
        &\frac{1}{2}\Big|g\big(N_J(\eta,(\nabla_{\zeta} J)\eta), J\zeta\big) \Big|\\
        &\leq \frac{1}{2}\Big|d\omega_J(\zeta,\eta,(\nabla_{\zeta} J)\eta)\Big|+\frac{1}{2}\Big|d\omega_J(\zeta,J\eta,J(\nabla_{\zeta} J)\eta)\Big|+\\
        &\frac{1}{2}\|N_J(\eta,(\nabla_{\zeta} J)\eta)\|_g \|J\zeta\|_g 
    \end{split}
\end{equation*}

By Lemma \ref{norm_bound}, the hermitianity of $g$ and the property $(\nabla_{\zeta} J)J\eta=-J(\nabla_{\zeta} J)\eta$, we thus have that 

\begin{equation*}
    \begin{split}
        \|(\nabla_{\zeta} J)\eta\|^2_g &\leq \frac{1}{2}\Big|d\omega_J(\zeta,\eta,(\nabla_{\zeta} J)\eta)\Big|+\frac{1}{2}\Big|d\omega_J(\zeta,J\eta,(\nabla_{\zeta} J)J\eta)\Big|+\\
        &\frac{1}{2}\|N_J(\eta,(\nabla_{\zeta} J)\eta)\|_g \|J\zeta\|_g\\
        &\leq \frac{\sqrt{3!}}{2}\epsilon\|\zeta\|_g \|\eta\|_g \|(\nabla_{\zeta} J)\eta\|_g +\frac{\sqrt{3!}}{2} \epsilon \|\zeta\|_g \|J\eta\|_g \|(\nabla_{\zeta} J)J\eta\|_g +\\
        &\frac{\sqrt{2!}}{2}\epsilon' \|\eta\|_g \|(\nabla_{\zeta} J)\eta\|_g \|J\zeta\|_g\\
        &=\sqrt{3!} \epsilon\|\zeta\|_g \|\eta\|_g \|(\nabla_{\zeta} J)\eta\|_g + \frac{\epsilon'}{\sqrt{2}} \|\zeta\|_g \|\eta\|_g \|(\nabla_{\zeta} J)\eta\|_g\\
        &=(\sqrt{6}\epsilon+\frac{\epsilon'}{\sqrt{2}})\|\zeta\|_g \|\eta\|_g \|(\nabla_{\zeta} J)\eta\|_g,
    \end{split}
\end{equation*} 

i.e. \[\|(\nabla_{\zeta} J)\eta\|^2_g \leq (\sqrt{6}\epsilon+\frac{\epsilon'}{\sqrt{2}})\|\zeta\|_g \|\eta\|_g \|(\nabla_{\zeta} J)\eta\|_g.\]

We claim that \[\|(\nabla_{\zeta} J)\eta\|_g \leq (\sqrt{6}\epsilon+\frac{\epsilon'}{\sqrt{2}})\|\zeta\|_g \|\eta\|_g\] for all vector fields $\zeta,\eta$. Indeed, if $\|(\nabla_{\zeta} J)\eta)\|_g = 0$, the claim reduces to the trivial fact  $(\sqrt{6}\epsilon+\frac{\epsilon'}{\sqrt{2}})\|\zeta\|_g \|\eta\|_g \geq 0.$ So assume that $\|(\nabla_{\zeta} J)\eta\|_g \neq 0$. Multiplying both sides by the function $\frac{1}{\|(\nabla_{\zeta} J)\eta\|_g}$ leads to the claimed inequality. Now, let $(e_i)_{i=1}^n$ be a local orthonormal frame of $T_X$. Then $\|(\nabla_{e_i} J)e_j\|_g \leq \sqrt{6}\epsilon+\frac{\epsilon'}{\sqrt{2}}$, so 

\begin{equation*}
    \begin{split}
        \|\nabla J \|^2_g &=\sum_{1\leq i,j\leq n} \|(\nabla_{e_i} J)e_j\|^2_g\\
        & \leq n^2 (\sqrt{6}\epsilon+\frac{\epsilon'}{\sqrt{2}})^2 
    \end{split}
\end{equation*}

Therefore, with $\epsilon''=n(\sqrt{6}\epsilon+\frac{\epsilon'}{\sqrt{2}})$, we deduce that \[\|\nabla J \|_g \leq \epsilon''.\]

For the converse, we use the formula in the proof of Lemma 1 of \cite{ACOTI}, namely \[N_J(\zeta,\eta)=J\big(d^{\nabla} J(J\zeta,J\eta)-d^{\nabla} J(\zeta,\eta)\big).\] Now, observe that for any vector fields $\zeta,\eta$, 

\begin{equation*}
    \begin{split}
        \|N_J(\zeta,\eta)\|_g &=\|d^{\nabla} J(J\zeta,J\eta)-d^{\nabla} J(\zeta,\eta)\|_g\\
        &\leq \|d^{\nabla} J(J\zeta,J\eta)\|_g + \|d^{\nabla} J(\zeta,\eta)\|_g\\
        &\leq \|(\nabla_{J\zeta} J)J\eta\|_g + \|(\nabla_{J\eta} J)J\zeta\|_g +\|(\nabla_{\zeta} J)\eta\|_g + \|(\nabla_{\eta} J)\zeta\|_g\\
        &\leq 4 \|\nabla J\|_g \|\zeta\|_g \|\eta\|_g \\
        &\leq 4\epsilon'' \|\zeta\|_g \|\eta\|_g.
    \end{split}
\end{equation*}
Here we used the fact that given a local orthonormal frame and vector fields $\zeta=\zeta^i e_i, \eta=\eta^j e_j$, setting $f_I=\zeta^i \eta^j$ and $v_I=(\nabla_{e_i} J) e_j$ for any $I\in \mathbb{N}^2_{\leq n}$, Cauchy-Schwarz gives that

\begin{equation*}
    \begin{split}
\|(\nabla_{\zeta} J)\eta\|^2_g &=\|\zeta^i \eta^j (\nabla_{e_i} J) e_j \|^2_g \\
&=\big\|\sum_{I \in \mathbb{N}^2_{\leq n}} f_I v_I \big\|^2_g \\
&\leq \big(\sum_{I \in \mathbb{N}^2_{\leq n}} |f_I|^2\big) \big(\sum_{I \in \mathbb{N}^2_{\leq n}} \|v_I\|^2_g\big)\\
&=\big(\sum_{1\leq i,j \leq n} |\zeta^i \eta^j|^2\big) \big(\sum_{1\leq i,j \leq n} \|(\nabla_{e_i} J) e_j\|^2_g\big)\\
&=\|\nabla J\|^2_g \|\zeta\|^2_g \|\eta\|^2_g
\end{split}
\end{equation*}

so that \[\|(\nabla_{\zeta} J)\eta\|_g \leq \|\nabla J\|_g \|\zeta\|_g \|\eta\|_g.\]
Hence, 

\begin{equation*}
    \begin{split}
        \|N_J\|^2_g &=\frac{1}{2!} \sum_{1\leq i,j\leq n} \|N_J(e_i,e_j)\|^2_g \\
        &\leq 8n^2 (\epsilon'')^2 ,
    \end{split}
\end{equation*}
and so with $\epsilon'=2\sqrt{2}\epsilon'' n$ we deduce that $\|N_J\|_g \leq \epsilon'.$

Finally, from the definition of differential and Lemma \ref{cov_kaehler}, it follows that $d\omega_J$ is also bounded. Indeed, since 

\begin{equation*}
    \begin{split}
        d\omega_J (\zeta,\eta,\rho)&=(\nabla_{\zeta} \omega_J)(\eta,\rho)-(\nabla_{\eta} \omega_J)(\zeta,\rho)+(\nabla_{\rho} \omega_J)(\zeta,\eta),
    \end{split}
\end{equation*}
we have that 
\begin{equation*}
    \begin{split}
        |d\omega_J (\zeta,\eta,\rho)|&\leq \big|g\big((\nabla_{\zeta} J)\eta,\rho\big) \big|+\big|g\big((\nabla_{\eta} J)\zeta,\rho\big) \big|+\big|g\big((\nabla_{\rho} J)\zeta,\eta\big) \big|\\
        &\leq 3\|\nabla J\|_g \|\zeta\|_g \|\eta\|_g \|\rho\|_g.
    \end{split}
\end{equation*}

Hence, 

\begin{equation*}
    \begin{split}
        \|d\omega_J\|^2_g&=\frac{1}{3!}\sum_{1\leq i,j,k\leq n} |d\omega_J (e_i,e_j,e_k)|^2\\
        &\leq \frac{1}{3!}\sum_{1\leq i,j,k\leq n} 9\|\nabla J\|^2_g\\
        &=\frac{3}{2}n^3 \|\nabla J\|^2_g\\
        &\leq \frac{3}{2}n^3 (\epsilon'')^2.
    \end{split}
\end{equation*}

So take $\epsilon=\sqrt{\frac{3}{2}n^3}\epsilon''$ to find that \[\|d\omega_J\|_g \leq \epsilon.\]
\end{proof}

\begin{definition}\label{KT_def2}
 Let $\epsilon>0.$ An almost-hermitian manifold $(X,g,J)$ is $\epsilon$-almost-complex-K{\"a}hler if $\|\nabla J\| \leq \epsilon.$ We say that $(X,g,J)$ is almost-complex-K{\"a}hler (acK) if it is $\epsilon$-acK for some constant $\epsilon.$ 
\end{definition}

Of course, K{\"a}hler manifolds are $\epsilon$-acK for all $\epsilon > 0$. In the other direction, a manifold that is $\epsilon$-acK for all $\epsilon >0$ is Kähler. Indeed $\|\nabla J\| \leq \epsilon$ for all positive $\epsilon$ is equivalent to $\nabla J = 0$.

\section{Compactness for acK manifolds}\label{3}

A $W^{2,p}\times W^{1,p}$ almost-hermitian manifold is a triple $(X,g,J)$ with $(X,g)$ a $W^{2,p}$ Riemannian manifold, and $J$ a $W^{1,p}$ almost-complex structure compatible with $g$ (i.e.\ pointwise $J$ is an orthogonal matrix). See Definition \ref{metric_regularity} and the paragraph directly below. We call the smooth structure of $X$ the {\em background structure}. Recall that if $p\gg 0$ is sufficiently big, we have the Sobolev embedding $W^{2,p}\subset C^{1,\alpha}$ for some $0<\alpha < 1$. With $p\gg 0$ we have $W^{1,p}\subset C^{0,\alpha}$, in particular the operator $J$ is continuously well-defined pointwise on $X$. In what follows and throughout the paper, the Sobolev exponent $p$ is 
chosen as $p\gg 0$ so that the Sobolev embedding theorems $W^{1,p} \subset C^0$ and various product theorems hold. 

The regularity conditions mean that in background smooth local coordinates, the matrix expressions for $g$ and $J$ have coefficients in $W^{2,p}$, respectively $W^{1,p}$ (that is to say the $g_{ij}$ are $W^{2,p}$ functions and the $J^j_i$ are $W^{1,p}$ functions). Polynomials in $W^{2,p}$ and $W^{1,p}$ remain in $W^{2,p}$ resp. $W^{1,p}$ by the Sobolev embeddings into continuous functions. Similar remarks hold for the inverse of a nowhere vanishing function. Thus, we can make the following observation.

\begin{remark}
\label{coeffs}
With respect to the background smooth structure, 
the matrix coefficients for the 
Hodge star operator $\star$ are in $W^{2,p}$, and
the matrix coefficients for $\nabla$ and $d^{\ast}$ acting on $k$-forms are in $W^{1,p}$. 
\end{remark}

\begin{definition}\label{ahconvergence_def}
A sequence of $n$-dimensional almost-hermitian manifolds $\{(X_i, g_i, J_i)\}$ converges in a $B \times B'$ topology (resp. weak $B\times B'$-topology) to an almost-hermitian manifold $(X,g,J)$, where $g$ is a $B$ Riemannian metric and $J$ is a $B'$ almost-complex structure, if for sufficiently large $i,$ there is a sequence of diffeomorphisms $\{f_i \colon X\to X_i\}$ such that $\{f^{*}_i g_i\}$ and $\{f^{*}_i J_i \coloneqq (f_i)^{-1}_* J_i (f_i)_* \colon T_X \to T_X\}$ converge (resp. converge weakly) on $X$ in the $B$ topology to $g$, and in the $B'$ topology to $J$.
\end{definition}

Theorem \ref{prepctness_thm}, whose proof will be given below, uses this definition. It says that there will be 
a convergent subsequence in the $W^{2,p}\times W^{1,p}$ topology with limiting $W^{2,p}\times W^{1,p}$ almost-hermitian manifold $(X_{\infty},g_{\infty},J_{\infty})$ meaning that $g_{\infty}$ is $W^{2,p}$ and $J_{\infty}$ is $W^{1,p}$. 

We will use the following standard functional analysis result.

\begin{theorem}\label{fa_thm}
Any bounded sequence in a reflexive Banach space has a weakly convergent subsequence. 
\end{theorem}

By Theorem \ref{fa_thm}, a bounded sequence $\{f_k\} \subset W^{1,p}$ (i.e.\ $\|f_k\|_{W^{1,p}} \leq C$ for some $C$ and all $k$) has a subsequence $\{f_{k_l}\}$ that converges weakly to some $f\in W^{1,p}.$ Moreover, we have that \[\|f\|_{W^{1,p}} \leq \liminf{\|f_{k_l}\|_{W^{1,p}}}=C.\] We certainly have the same statement for $L^p$ bounded sequences.

\noindent
\emph{Proof of Theorem \ref{prepctness_thm}}. Let $a>0,$ $(\epsilon_i)_{i=1}^{\infty} \subset (0,a]$ be a sequence and $\epsilon \coloneqq \liminf{\epsilon_i}$. We first show that any sequence $\{(X_i,g_i,J_i)\}_{i=1}^{\infty}$ of respectively $\epsilon _i$-acK manifolds in $\mathcal{A}(n,d,k,i_0,a)$ has a subsequence $\{(X_{i_l},g_{i_l},J_{i_l})\}_{l=1}^{\infty}$ that converges in the weak-$W^{2,p} \times W^{1,p}$ topology to an almost-hermitian manifold $(X_{\infty},g_{\infty},J_{\infty})$ satisfying $\| \nabla^{\infty} J_{\infty} \| _{L^p} \leq \epsilon,$ where $\nabla^{\infty}$ is the Levi-Civita connection of the limit metric and the $L^p$ norm is the one induced by that metric. See the final paragraph to understand why this is enough to conclude that the limit belongs to $\mathcal{A}(n,d,k,i_0,a).$

By Theorem \ref{Acpctness}, we obtain a subsequence $\{(X_{i_l},g_{i_l})\}_{l=1}^{\infty}$ of $\{(X_i,g_i)\}_{i=1}^{\infty}$ converging in the weak-$W^{2,p}$ topology to a $W^{2,p}$ Riemannian manifold $(X_{\infty},g_{\infty}).$ We are assuming that for large enough $i,$ $\|\nabla^{g_i} J_i\|_{g_i} \leq \epsilon,$ which implies that $\|\nabla^i f^*_i J_i\|_{f^*_i g_i} \leq \epsilon,$ where $\nabla^i$ is the Levi-Civita connection of the Riemannian metric $f^{*}_i g_i$ on $X_{\infty}$.
Notice that $J_i$ is also bounded in norm, since it is pointwise an orthogonal matrix with respect to $g_i$ and the $g_i$ are all equivalent since they approach a limit strongly in $C^1$. Therefore, $J_i$ is bounded in the $W^{1,p}$ norm. The weak(*) and weak topologies in $W^{1,p}$ are the same, so we can extract a subsequence $J_{i_l}$ that converges weakly in $W^{1,p}$ to a limit $J_{\infty}$ which is a $W^{1,p}$ almost-complex structure compatible with the limit $g_{\infty}$. We can apply Theorem \ref{fa_thm} to the sequence $\{\nabla^i f^*_i J_i\}$ to deduce that the subsequence $\{\nabla^{i_l} f^*_{i_l} J_{i_l}\}$ converges weakly in $L^p,$ and the limit is also bounded by $\epsilon$, i.e.\  $\| \nabla^{\infty} J_{\infty} \| _{L^p} \leq \epsilon.$

We now justify that a.e. $\| \nabla^{\infty} J_{\infty} \| _{g_{\infty}} \leq \epsilon$ as well, so in this sense $(X_{\infty},g_{\infty},J_{\infty}) \in \mathcal{A}(n,d,k,i_0,a)$. We do this by looking at the uniformity in terms of $p$. One should note however that the discussion of \cite{Anderson_notes} requires using a fixed value of $p$ since the lower bound for the $W^{2,p}$ harmonic radius may depend on $p$. Consider some $p'>p$, and extract a subsequence having a weak $W^{2,p'}\times W^{1,p'}$ limit 
$(X'_{\infty},g'_{\infty},J'_{\infty})$ 
according to the first part of this proof applied to $p'$.  
Now see the paragraph after the statement of 
Theorem \ref{Acpctness} to compare the two limiting manifolds. The limiting image by the harmonic coordinates for both limiting processes provides the graph of a $W^{3,p}$ isomorphism $X_{\infty}\cong X'_{\infty}$ which we may assume respects both $g$ and $J$. Use it to transport the bound 
$\| \nabla^{\prime , \infty} J'_{\infty} \| _{L^{p'}} \leq \epsilon$ from $X'_{\infty}$ back to $X_{\infty}$. This gives the bound $\| \nabla^{\infty} J_{\infty} \| _{L^{p'}} \leq \epsilon$. That bound is therefore valid for any $p'$. An $L^{p'}$ bound that is uniform in $p'\rightarrow \infty$ implies the same bound in $L^{\infty}$, so we get the required 
$\| \nabla^{\infty} J_{\infty} \| _{g_{\infty}} \leq \epsilon$ 
as a bound on this measurable function, almost everywhere. 
\QEDB

\section{$L^2$ cohomology}\label{5}

Consider a $W^{2,p}$ Riemannian manifold 
$(X,g)$. We would like to obtain a Hodge decomposition result in this setting. We are going to use $L^2$ cohomology involving forms with a lower level of regularity than the ambient space. We rely, initially, on the main results for $L^2$ cohomology of the smooth background manifold $X$, recalling that the exterior differential does not depend on $g$. 

Introduce the following notation: $A^k[\ldots ]$ will indicate the space of global sections of $k$-forms on $X$, with regularity and other conditions in the brackets. Thus $A^k[L^2]$ denotes the space of $k$-forms with $L^2$ coefficients. These will be considered as subspaces of the space of distribution-valued $k$-forms, so we may apply differential operators to get other distribution-valued forms or sections. 

\subsection{The $L^2$ cohomology complex}\label{5.1}

In this subsection everything will refer to the smooth background manifold. 

Define the $L^2$ domain of $d$ by
\[
A^k[L^2,Dom(d)] \coloneqq
\{ \alpha \in A^k[L^2] \mid d(\alpha ) \in A^{k+1}[L^2] \} .
\] 
We have 
\[
A^k[L^2,Dom(d)]
\stackrel{d}{\rightarrow}
A^{k+1}[L^2,Dom(d)] 
\]
indeed for $\alpha \in A^k[L^2,Dom(d)]$ then 
$d(d(\alpha )) = 0$ (true with distribution coefficients) so $d(\alpha )$, which 
{\em a priori} is in $A^k[L^2]$, may also be seen as an element of 
$A^{k+1}[L^2,Dom(d)] $.

We get a complex of vector spaces
\[ 
A^{\cdot }[L^2,Dom(d)] := 
\left( 
\ldots \stackrel{d}{\rightarrow}
A^k[L^2,Dom(d)]
\stackrel{d}{\rightarrow}
\ldots \right) . 
\]
This complex maps into the usual complex of smooth forms 
$A^{\cdot }[C^{\infty}]$.

\begin{theorem}
\label{classicL2}
The map from 
$A^{\cdot }[L^2,Dom(d)] $ to 
$A^{\cdot }[C^{\infty}]$ induces an isomorphism of cohomology, in other words the cohomology of the complex $A^{\cdot }[L^2,Dom(d)] $ is naturally isomorphic to the De Rham cohomology of $X$. The image $dA^{\cdot }[L^2,Dom(d)] $ is a closed subspace of 
$A^{\cdot }[L^2]$.
\end{theorem}

Theorem \ref{classicL2} is the base case of $L^2$ cohomology theory, for a compact smooth manifold. It is subsumed in the many references on this topic, we note for example  
\cite{Aubin, Carron, CKS, CheegerGoreskyMacPherson, DodziukL2, DodziukSobolev, MitreaMitrea,
Schmid, Zucker}.

\subsection{The domain of $d$}\label{5.2}

It will be useful to have a more precise description of the domain of $d$. 
In this subsection we again refer to the smooth background manifold, up until the discussion of the domain of $d^{\ast}$ at the end.  

\begin{lemma}
\label{domd}
The domain of $d$ may be expressed as a sum (though this is not a direct sum)
\[
A^k[L^2,Dom(d)] = A^k[L^2,\ker (d)] + A^k[W^{1,2}]
\]
where $A^k[L^2,\ker (d)]$ is the space of $\alpha \in A^k[L^2]$ with $d(\alpha ) = 0$ and $A^k[W^{1,2}]$ is the space of $k$-forms with coefficients in the Sobolev space $W^{1,2}$. Furthermore we have a direct sum decomposition
\[ 
A^k[L^2,\ker (d)] = \mathcal{H}^k_{bk} \oplus 
d\left(  A^{k-1}[W^{1,2}] \right) .
\]
Here $\mathcal{H}^k_{bk}$ is the finite-dimensional space of harmonic forms for the smooth background metric, and recall that these harmonic forms are smooth. 
\end{lemma}
\begin{proof}
Working with the smooth background metric we may apply the theory of the Green's operator $G$. If $\alpha \in A^k[L^2]$ then it may be decomposed as
\[
\alpha = a + \Delta _{d,bk} G (\alpha )
\]
and $G(\alpha ) \in A^k[W^{2,2}]$. Thus
\[
\alpha = \alpha ' + \alpha ''
\]
with 
\[
\alpha ' := a+dd^{\ast}_{bk}G(\alpha )\in A^k[L^2,\ker (d)]
\]
and 
\[
\alpha '' :=  d^{\ast}_{bk} d G(\alpha ) \in A^k[L^2,\ker (d^{\ast}_{bk})].
\]
Suppose now that $\alpha \in A^k[L^2,Dom(d)]$. It means that there is
$\eta \in  A^{k+1}[L^2]$ with $d(\alpha ) = \eta$ as distributions. 
Then $d(\alpha ')=0$  gives $d(\alpha '') = \eta$. 

One may now use the statement that $d+d^{\ast}_{bk}$ is a first-order elliptic operator
whose value on $\alpha ''$ is $\eta \in A^{k+1}[L^2]$. Elliptic regularity implies that
$\alpha '' \in A^k[W^{1,2}]$. 

Another way of saying this is with the slightly more classical (although using negative Sobolev spaces) regularity for the second order elliptic operator $\Delta_{d,bk}$ as follows. 
As distributions, 
\[
\Delta _{d,bk}(\alpha '') = d^{\ast}_{bk}(\eta ) \in A^k[W^{-1,2}]
\]
so second-order elliptic regularity gives $\alpha '' \in A^k[W^{1,2}]$.

We have decomposed $\alpha \in A^k[L^2,Dom(d)]$ as a sum 
$\alpha = \alpha ' + \alpha ''$ with $\alpha ' \in A^k[L^2,\ker (d)]$ and
$\alpha '' \in A^k[W^{1,2}]$. This is the first statement of the lemma. 

The second statement follows from the Hodge decomposition theorem for the smooth background metric.
\end{proof}

Now we consider the operator $d^{\ast}$ that depends on the $W^{2,p}$ Riemannian metric $g$ on $X$. If $\alpha \in A^k[L^2]$ then $\star \alpha \in A^{n-k}[L^2]$ and 
using negative Sobolev spaces, 
$d(\star \alpha ) \in A^{n-k}[W^{-1,2}]$. Since $\star$ is in $W^{2,p}\subset C^1$, multiplication by $\star$ preserves the $W^{-1,2}$ negative Sobolev space (one may argue using distributions as functionals on test-functions here). This allows us to define $d^{\ast}\alpha := -\ast d\ast \alpha$ as a distribution. 

\begin{lemma}
\label{domdstar}
The domain of $d^{\ast}$ may be expressed as a sum (though this is not a direct sum)
\[
A^k[L^2,Dom(d^{\ast})] = A^k[L^2,\ker (d^{\ast})] + A^k[W^{1,2}]
\]
where $A^k[L^2,\ker (d^{\ast})]$ is the space of $\alpha \in A^k[L^2]$ with $d^{\ast}(\alpha ) = 0$ and $A^k[W^{1,2}]$ is the space of $k$-forms with coefficients in the Sobolev space $W^{1,2}$. Furthermore we have a direct sum decomposition
\[ 
A^k[L^2,\ker (d^{\ast})] = \star \mathcal{H}^{2n-k}_{bk} \oplus 
d^{\ast}\left(  A^{k-1}[W^{1,2}] \right) .
\]
The forms in $\star \mathcal{H}^{2n-k}_{bk}$ are in $A^k[W^{2,p}] \subset 
A^k[C^1]$.
\end{lemma}
\begin{proof}
Since $d^{\ast} = -\star d \star$ we get this statement by conjugating the previous one with the Hodge star $\star$. Recall that $\ast$ has $W^{2,p}$ coefficients and these are in $C^1$. 
\end{proof}

\subsection{Stokes' theorem}\label{5.3}

In this section we look at the statement 
$(d(\alpha ), \beta )_{L^2}  = (\alpha , d^{\ast}(\beta ))_{L^2} $
for the formal adjoint $d^{\ast}$ defined using the $W^{2,p}$ 
Riemannian metric $g$.

\begin{proposition}
\label{stokes}
Suppose 
$\alpha \in A^{k}[L^2,Dom(d)]$ and $\beta \in A^{k+1}[L^2,Dom(d^{\ast})]$.
Then 
\[
(d(\alpha ), \beta )_{L^2} =  
(\alpha ,d^{\ast}( \beta  ))_{L^2} .
\]
\end{proposition}
\begin{proof}
Suppose $\alpha$ and $\beta$ are smooth forms with respect to the background 
structure. Assume they have real coefficients just to avoid complex conjugation in the notation. 
Recall that 
\[ 
(d(\alpha ), \beta )_{L^2} = 
\int_X d(\alpha ) \wedge \star \beta .
\]
Multiplication by $\alpha$ and $d\alpha$ is well defined and linear on distributions. We have that $\star \beta \in A^{\cdot}[W^{2,p}]$. In particular, it is a distribution and one can differentiate it, so we can say 
\[
0 = \int d(\alpha \wedge \star \beta ) = 
\int d(\alpha ) \wedge \star \beta + \int \alpha \wedge d(\star \beta ),
\]
so 
\[ 
\int d(\alpha ) \wedge \star \beta = -  \int \alpha \wedge \star (\star d(\star \beta ))
= (\alpha , d^{\ast} \beta )_{L_2}.
\]
This proves the proposition when the forms have smooth coefficients. 

Now, suppose $\alpha \in A^{k}[L^2,Dom(d)]$ and $\beta \in A^{k+1}[L^2,Dom(d^{\ast})]$. From the above, we may write
\[ 
\alpha = a + \alpha ' + d(\alpha '')
\]
where $a \in \mathcal{H}^k_{bk}$, $\alpha '\in  A^{k}[W^{1,2}]$
and $\alpha'' \in  A^{k-1}[W^{1,2}]$. Similarly we may write
\[
\beta = b + \beta ' + d^{\ast} (\beta '')
\]
where $b\in \star \mathcal{H}^{2n-k-1}_{bk}$, $\beta '\in  A^{k+1}[W^{1,2}]$
and $\beta''\in  A^{k+2}[W^{1,2}]$.
For the last two parts we note that $\star$ is $C^1$ so it preserves the $W^{1,2}$ Sobolev spaces.

Break the adjointness statement into statements for the different pieces of $\alpha$ and $\beta$. We have $d(a + \alpha ' + d(\alpha '') )= d(\alpha ')$ and
$d^{\ast} ( b + \beta ' + d^{\ast} (\beta '')) = d^{\ast}(\beta ')$. We claim that
\[
(d(\alpha '),b)_{L^2} = 0, \;\;\;\; (d(\alpha '),d^{\ast} (\beta ''))_{L^2} = 0
\]
and similarly 
\[
(a,d^{\ast} (\beta '))_{L^2} = 0, \;\;\;\; (d(\alpha ''),d^{\ast} (\beta '))_{L^2} = 0.
\]
For the terms involving $a$ and $b$ note that $a$ and $\star b$ are smooth forms and wedging with them is a linear operation on distributions. The functions 
\[
\alpha ' \mapsto (d(\alpha '),b)_{L^2}  \mbox{ and }
\beta ' \mapsto (a,d^{\ast} (\beta '))_{L^2}
\]
are bounded linear functions on the $W^{1,2}$ spaces. Since smooth forms are dense in $W^{1,2}$ (the Meyers-Serrin theorem) 
and these functions vanish on smooth forms, they vanish. To complete the claim, the two remaining terms are the same: the function 
\[
\eta , \zeta \mapsto (d(\eta ),d^{\ast} (\zeta ''))_{L^2}
\]
is a continuous bilinear form on $A^{k-1}[W^{1,2}] \times 
A^{k+1}[W^{1,2}]$, and it vanishes on pairs of smooth forms. Since the smooth forms are dense, we get $(d(\eta ),d^{\ast} (\zeta ''))_{L^2}=0$ as claimed. 

Using the claims, the only remaining formula that needs to be shown is the statement
of the proposition but for $W^{1,2}$ forms, that is:
\[
(d(\alpha ' ), \beta ' )_{L^2} -  
(\alpha ' ,d^{\ast}( \beta  ' ))_{L^2} =0
\]
for $\alpha  ' \in A^{k}[W^{1,2}]$ and $\beta ' \in A^{k+1}[W^{1,2}]$.
But now the function 
\[
\alpha ', \beta ' \mapsto 
(d(\alpha ' ), \beta ' )_{L^2} -  
(\alpha ' ,d^{\ast}( \beta  ' ))_{L^2}
\]
is a continuous bilinear function on 
$A^{k}[W^{1,2}] \times A^{k+1}[W^{1,2}]$, vanishing on smooth forms, so again by density it vanishes. This completes the proof. 
\end{proof}

\begin{corollary}
\label{laplacian-norm-squared}
Suppose $\alpha \in A^k[W^{2,2}]$. Then 
\[
\| d(\alpha ) \| ^2 _{L^2} + \| d^* (\alpha ) \| ^2 _{L^2} = \int _X (\Delta _d \alpha , \alpha )_g dVol .
\]
In particular, if $\Delta _d(\alpha ) = 0$ then $d(\alpha ) = 0$ and 
$d^{\ast} (\alpha ) = 0$. 
\end{corollary}
\begin{proof}
We note that $\Delta _d \alpha$ in the distributional sense, is in 
$A^k[L^2]$ so the integral on the right hand side makes sense. We have 
$\alpha \in A^k[L^2,Dom(d)]$, and
$d(\alpha ) \in A^{k+1}[W^{1,2}]\subset A^{k+1}[L^2,Dom(d)]$ and $d^{\ast}(\alpha ) 
\in A^{k-1}[W^{1,2}]\subset A^{k-1}[L^2,Dom(d)]$ so Proposition \ref{stokes} may be applied:
\[
(d^{\ast}d\alpha , \alpha )_{L^2} = (d \alpha , d \alpha )_{L^2}
= \| d\alpha \| ^2_{L^2}
\]
and
\[
(dd^{\ast}\alpha , \alpha )_{L^2} = (d^{\ast} \alpha , d^{\ast} \alpha )_{L^2}
= \| d^{\ast}\alpha \| ^2_{L^2}.
\]
Adding these together gives the required statement. 
\end{proof}

\subsection{Less regular products and the Leibniz formula}
\label{lessreg}

It will also be useful to apply the product differentiation rule to products with a lower level of regularity, including negative Sobolev spaces. We adopt the informal convention that the negative Sobolev space $W^{-1,q}$ consists of distributions that are locally of the form $v=D_iu_i$ for $u_i\in L^q$. Suppose $f\in W^{1,p}$ with $L^p\cdot L^q\subset L^1$. We can then define the product 
$fv$ as a distribution in the following way: if $\varphi$ is a smooth test function then set
\[
\langle fv,\varphi \rangle \coloneqq \langle D_i(fu_i),\varphi \rangle 
- \langle D_i(f)u_i,\varphi \rangle 
\]
where the two terms on the right are defined, for the second one 
because $D_i(f)u_i \in L^p\cdot L^q \subset L^1$, and for the first by putting $D_i$ over onto $v$ and noting that the product is even better defined. 

This definition gives tautologically a Leibniz rule for differentiating products. 

\subsection{The Hodge decomposition}\label{5.4}

Define subspaces of $A^k[L^2]$
\[
Im(d) := d A^{k-1} [L^2,Dom(d)], 
\;\;\;\;\;\;
Im(d^{\ast}) := d^{\ast} A^{k+1} [L^2,Dom(d^{\ast})].
\]

\begin{corollary}
\label{decompcor}
The subspaces $Im(d)$ and $Im(d^{\ast}) $ are closed in $A^k[L^2]$. They have expressions 
\[
Im(d) = d A^{k-1} [W^{1,2}], 
\;\;\;\;\;\;
Im(d^{\ast}) = d^{\ast} A^{k+1} [W^{1,2}].
\]
Their perpendicular subspaces with respect to the inner product $(\; , \; )_{L^2}$ are
\[
Im(d) ^{\perp} = A^k[L^2,\ker (d^{\ast}) ],
\]
\[
Im(d^{\ast}) ^{\perp} = A^k[L^2,\ker (d) ] .
\]
\end{corollary}
\begin{proof}
The statement that $Im(d)$ is closed is from Theorem \ref{classicL2}. The statement that 
$Im(d^{\ast}) $ is closed is obtained by taking the image with the Hodge star $\star$. 
The expressions as images of the $W^{1,2}$ subspaces come from Lemmas \ref{domd} and \ref{domdstar}. 

Suppose $\beta \in Im(d)^{\perp}$ in $A^k[L^2]$. Then for 
$\alpha \in A^{k-1}[C^{\infty}]$ we have, using Proposition \ref{stokes},
\[
(\alpha , d^{\ast}(\beta ))_{L^2} = 
(d(\alpha ), \beta )_{L^2} = 0.
\]
As $A^{k-1}[C^{\infty}]\subset A^{k-1}[L^2]$ is dense we get 
$(\alpha , d^{\ast}(\beta ))_{L^2} =0$ for all $\alpha \in A^{k-1}[L^2]$, thus
$d^{\ast}(\beta ) = 0$. This shows that 
\[
Im(d) ^{\perp} \subset  A^k[L^2,\ker (d^{\ast}) ].
\]
In the other direction, suppose $\beta  \in A^k[L^2,\ker (d^{\ast}) ]$. Note that automatically $\beta  \in A^k[L^2,Dom (d^{\ast}) ]$. Then 
for $\alpha \in Im(d)$ we can write $\alpha = d(\alpha ')$ for $\alpha '\in A^{k-1}[L^2,Dom (d)]$. Proposition \ref{stokes} gives

\[
(\alpha , \beta ) _{L^2} = (d(\alpha ') , \beta ) _{L^2} = ( \alpha ' , d^* (\beta ) )_{L^2} = 0.
\]

This shows that 
\[
A^k[L^2,\ker (d^{\ast}) ] \subset Im(d) ^{\perp} .
\]
We conclude that $Im(d) ^{\perp} =  A^k[L^2,\ker (d^{\ast}) ]$.

The proof that $Im(d^{\ast}) ^{\perp} = A^k[L^2,\ker (d) ] $ is similar, using
as necessary the $W^{2,p}$ regularity of $\star$ as we have done above. 
\end{proof}

We now prove the Hodge decomposition theorem in our situation. Recall that the space of harmonic forms for the $W^{2,p}$ metric $g$ is defined by
\[
\mathcal{H}^k_g := A^k[L^2,\ker (d) ]\cap A^k[L^2,\ker (d^{\ast}) ].
\]

\begin{remark}
\label{orrthoclosed}
The orthogonal direct sum of closed subspaces of a Hilbert space  is a closed subspace.
\end{remark}

\begin{theorem}
\label{hodgedecomp}
We have a direct sum decomposition, orthogonal for the inner product 
$(\; , \; )_{L^2}$:

\begin{equation*}
    \begin{split}
        A^k[L^2] &= \mathcal{H}^k_g \oplus Im(d) \oplus Im(d^{\ast})\\
                 &= \mathcal{H}^k_g \oplus d A^{k-1}[W^{1,2}] \oplus d^{\ast} A^{k+1}[W^{1,2}].
    \end{split}
\end{equation*}

The pieces may be regrouped as
\[
A^k[L^2,\ker (d) ] = \mathcal{H}^k_g \oplus Im(d)
\]
and 
\[
A^k[L^2,\ker (d^{\ast}) ] = \mathcal{H}^k_g \oplus Im(d^{\ast}),
\]
and also 
\[
A^k[L^2] = Im(d) \oplus A^k[L^2,\ker (d^{\ast}) ]
= A^k[L^2,\ker (d) ] \oplus Im (d^{\ast}).
\]
\end{theorem}
\begin{proof}
By definition $\mathcal{H}^k_g$ is a closed subspace contained in 
$A^k[L^2,\ker (d) ]$. We remark in passing that it is finite dimensional, 
see the corollary below. 

By Remark \ref{orrthoclosed} applied twice, the subspace
\[
\mathcal{H}^k_g \oplus Im(d) \oplus Im(d^{\ast}) \subset A^k[L^2] 
\]
is closed. The inclusion map from the  orthogonal direct sum into 
$A^k[L^2] $ is injective since we know that the subspaces are orthogonal. 
By the Hahn-Banach and Riesz representation theorems, if this subspace 
were not equal to $A^k[L^2]$ then there would be a vector 
$\eta \in A^k[L^2]$ perpendicular to all three subspaces. In particular, 
$\eta \in Im(d)^{\perp}$ but 
$Im(d) ^{\perp} = A^k[L^2,\ker (d^{\ast}) ]$. Similarly, 
$\eta \in Im(d^{\ast})^{\perp}$ but 
$Im(d^{\ast}) ^{\perp} = A^k[L^2,\ker (d) ]$. Thus, 
\[
\eta \in 
A^k[L^2,\ker (d^{\ast}) ] \cap  A^k[L^2,\ker (d) ]
= \mathcal{H}^k_g;
\]
but by hypothesis, $\eta$ is perpendicular to $\mathcal{H}^k_g$. Thus, $\eta = 0$. 
This shows the first statement that $A^k[L^2]$ equals $\mathcal{H}^k_g \oplus Im(d) \oplus Im(d^{\ast})$. The other statements follow in a similar way. 
\end{proof}

\begin{corollary}\label{Hodge_thm}
The map 
\[ 
\mathcal{H}^k_g \rightarrow 
\frac{
A^k[L^2,\ker (d) ]
}{
Im(d)
}
\cong H^k_{DR}(X)
\]
is an isomorphism. 
\end{corollary}
\begin{proof}
One can prove finite-dimensionality of the space of harmonic forms from before the proof of the previous theorem, as was mentioned above. Notice that $\mathcal{H}^k_g $  is also contained in 
$A^k[L^2,\ker (d^{\ast}) ]$, it is perpendicular to 
$Im (d)$ by Corollary \ref{decompcor}. In particular the map 
\[ 
\mathcal{H}^k_g \rightarrow 
\frac{
A^k[L^2,\ker (d) ]
}{
Im(d)
}
\]
is injective, so $\mathcal{H}^k_g$ is finite-dimensional since 
the $L^2$ cohomology (c.f.\ Theorem \ref{classicL2})
identifies the quotient $A^k[L^2,\ker (d) ]/Im(d)$ with the
usual finite dimensional 
De Rham cohomology of the compact smooth manifold $X$.

Using the direct sum decomposition formulas of the previous theorem gives surjectivity of the map. 
\end{proof}

\section{Regularity of harmonic forms}\label{6}

We would like to study the regularity properties of the harmonic forms 
$a\in \mathcal{H}^k_g$. The theory of Gilbarg-Trudinger 
\cite[Chapter 8]{GilbargTrudinger} on regularity of solutions of second order elliptic differential equations with non-smooth coefficients provides an important input. However, it requires the solutions to be measurably differentiable functions.

\subsection{Regularity for second order elliptic equations}\label{6.1}
The classical reference is Gilbarg-Trudinger \cite{GilbargTrudinger}, and we use Beck 
\cite{Beck} for the case of vector-valued equations. We record here a version of their statement that will be most useful for our scenario.

\begin{lemma}
\label{gt}
Suppose $L$ is a second order differential operator on a vector bundle $V$ 
(assume with a frame) over an
 open subset of $\mathbb{R} ^{n}$ of the form
\[
L u = a_{ij} D_i D_j u + b_i D_i u + c u
\]
where $a_{ij}$, $b_i$ an $c$ are matrix-valued functions (the matrices corresponding to the action on $V$), such that for some $p\gg 0$, 
\[
a_{ij} \in W^{2,p}, \;\;\;\; 
b_i \in W^{1,p}, \;\;\;\; 
c \in L^p .
\]
Suppose the matrix of leading coefficients $a_{ij}$ strictly satisfies the ellipticity condition. 
If $v\in L^2$ and $u\in W^{1,2}$ is a solution of the inhomogeneous equation
$Lu=v$, in the sense of distributions, 
then $u\in W^{2,2}$ after shrinking to a smaller open set. 
\end{lemma}
\begin{proof}
Note that $A = (a_{ij})\in C^{1,\alpha}$ and $b_i \in C^{0,\alpha}$ by the Sobolev embeddings for $p$ big enough and corresponding $\alpha \in (0,1)$. 
For any $\varepsilon >0$ with $(n-2)\varepsilon < 2$, 
we have $W^{1,2} \subset L^{2+\varepsilon}$. Then for 
$p$ big enough we have the multiplication $L^p \cdot L ^{2+\varepsilon}$.
Thus, if $u\in W^{1,2}$ it gives $cu \in L^2$. 
Note similarly that  $b_iu\in W^{1,2}$, 
indeed $ (\partial _jb_i)u\in L^2$ as above, and 
$b_i\in C^0$ so $b_i(\partial _j u) \in L^2$. Similarly $div(A)u \in W^{1,2}$.

Write the equation in divergence form 
\cite[Equation (8.1)]{GilbargTrudinger} \cite[Equation (4.2)]{Beck}
\[
div(A D u) = div (f) + g 
\]
where $A=(a_{ij})$ and $f = (b-div(A))u$ with $g$ containing the
inhomogeneous part $v$,  the piece $cu$ and the remaining terms coming from the transformation to divergence form. The desired regularity $u\in W^{2,2}$ is given by 
\cite[Theorem 8.8]{GilbargTrudinger}, although that concerns the scalar case. 
For the vector valued case see Beck \cite[Theorem 4.9]{Beck} (taking $k=1$)  which shows that at our level of regularity (with $A\in W^{2,p}\subset C^{1,\alpha}$, $f\in W^{1,2}$ and $g\in L^2$) we don't run into the counterexamples of De Giorgi \cite{DeGiorgi}. 
\end{proof}

\begin{proposition}
\label{gt-laplacian}
For a Riemannian manifold with $W^{2,p}$ metric $g$:
\begin{enumerate}

\item the Laplacian operator $\Delta _d = dd^{\ast} + d^{\ast}d$ has the form 
of \cite[(8.1)]{GilbargTrudinger}  as well as the form required for entering into Lemma \ref{gt}; 

\item for $\alpha \in A^k[W^{1,2}]$, then 
$\Delta _d\alpha \in A^k[W^{-1,2}]$ is a distribution that may be computed in terms of the expression with $\star$ and $d$; 

\item  if $\alpha \in A^k[W^{1,2}]$
is a weak solution of an inomogeneous equation $\Delta _d \alpha = \eta$ with $\eta \in A^k[L^2]$, then $\alpha \in A^k[W^{2,2}]$.
\end{enumerate}
\end{proposition}
\begin{proof}
Use local coordinates for the smooth background structure, and write $d$ as an operator 
$d=m_iD_i$ (summing on repeated indices) where, in this case, there is no zeroth order term and $m_i$ is a matrix operator with constant coefficients on 
$\bigoplus \bigwedge ^k T^{\ast}X$ (since we use the $dx_j$ and $dx_I$ for multi-indices $I$ as frames). 
Similarly write $\star = s$ as a matrix operator with $W^{2,p}$ coefficients. 
Then formally
\[
d^{\ast}d \alpha = -s m_iD_i (s m_jD_j( \alpha )) .
\]
We may transform this into the following ``divergence form'' 
\cite[(8.1)]{GilbargTrudinger} 
\[
d^{\ast}d \alpha = - D_i (sm_i sm_jD_j (\alpha ) )
\]
\[
+ (D_i(sm_i)\cdot sm_j)D_j (\alpha ).
\]
This may also be transformed into the form stated in Lemma \ref{gt}. 

For $\alpha \in A^k[W^{1,2}]$, the statement that $d^{\ast}d\alpha =\beta$ as a distribution means that for a test function $v$ (compactly supported in the coordinate neighborhood), 
\[ 
\langle \beta , v\rangle = 
\langle (sm_i sm_jD_j \alpha ) ,D_i v \rangle 
+ \langle \beta ',v\rangle 
\]
where $\beta ':= (D_j(sm_j)\cdot sm_i )D_i \alpha$.

We note that $\star d \alpha \in L^2$ so $d(\star d \alpha ) \in W^{-1,2}$ is a first derivative of an $L^2$ function. If $f\in W^{2,p}$ and $g_i\in L^2$ then for a smooth compactly supported test function, $D_i(fv) \in W^{1,p}\subset C^0$ can be paired with $g_i\in L^2$, so proceeding as in Section \ref{lessreg}
we can define the distribution $f\cdot (D_i g_i)$ by 
\[
\langle f\cdot (D_i g_i), v\rangle := - \langle g_i, D_i(fv)\rangle . 
\] 
Applied to the coefficients, this allows us to define the distribution $\star d(\star d \alpha )$. This pairs on smooth test functions in the same way as defined previously, noting that since $m_i$ are constant and $s$ is in $W^{2,p}$, the rules for differentiation of products necessary for that computation hold. 

This gives (1) and (2) for the term $d^{\ast}d$ and the term 
$dd^{\ast}$ is treated similarly. 

For (3), apply the theory of \cite[Chapter 8]{GilbargTrudinger} \cite{Beck} expressed in Lemma \ref{gt}, noting that a weak (i.e.\ distributional) solution is a weak solution in the sense of \cite{GilbargTrudinger}. We get the required regularity. 
\end{proof}

\subsection{The case of harmonic forms}\label{6.2}

Recall that the definition of the space of harmonic forms is via the conditions $d(\alpha ) = 0$ and $d^{\ast}(\alpha ) = 0$. 

Suppose $\alpha \in A^k[L^2, Dom(d)] \cap A^k[L^2, Dom(d^{\ast})]$. Recall this means there are forms $\beta \in A^{k+1}[L^2]$ and $\beta ' \in A^{k-1}[L^2]$ 
such that, in the sense of distributions on the smooth background manifold, $d(\alpha )=\beta$ and $d(\star \alpha ) = \star \beta '$. With this understanding of what it means distributionally, the second condition may be written as $d^{\ast}(\alpha ) = -\beta '$.

\begin{proposition}
\label{reg12}
If $\alpha \in A^k[L^2, Dom(d)] \cap A^k[L^2, Dom(d^{\ast})]$ is harmonic,
then $\alpha \in A^k[W^{1,2}]$. In particular, the coefficients of $\alpha$ are functions with measurable derivative.
\end{proposition}
\begin{proof}
The proof is by induction on $k$. Notice that when $k=0$, $\alpha$ is a function and 
$d(\alpha )$ contains all the partial derivatives of $\alpha$ so the conclusion is automatic. 

Suppose $k\geq 1$ and suppose we know the proposition for smaller values of $k$.  
By the $L^2$ cohomology theory, $\alpha$ represents a class in $H^k_{DR}(X)$. 
This class has a smooth representative 
$\eta\in a^k[C^{\infty}]$, so $\alpha - \eta \in Im(d)$, that is to say there is $\zeta \in A^{k-1}[L^2,Dom(d)]$ such that $d(\zeta ) = \alpha - \eta$. 

We may assume that $d^{\ast}(\zeta ) = 0$. Indeed if $\zeta _1$ is a first choice in $A^{k-1}[L^2,Dom(d)]$ with $d(\zeta _1) = \alpha - \eta$
then we may decompose using Theorem \ref{hodgedecomp} as 
\[ 
\zeta _1 = \zeta _0 + \zeta 
\]
with $\zeta _0 \in Im(d)$ and $\zeta \in \ker (d^{\ast})$. Now $d(\zeta ) = \alpha - \eta$ and $d^{\ast}\zeta = 0$. 

By the inductive statement of the proposition which holds for $k-1$, we get $\zeta \in A^{k-1}[W^{1,2}]$. Now write
\[
\Delta _d(\zeta ) = (dd^{\ast} + d^{\ast} d) \eta = d^{\ast} ( \alpha - \eta ) 
\]
giving 
\[
\Delta _d(\zeta ) = \beta '' := \beta ' - d^{\ast}(\eta ) \in A^{k-1}[L^2].
\]
Thus $\zeta$ is a $W^{1,2}$ distributional solution of 
the inhomogeneous equation $\Delta _d(\zeta ) = \beta ''$. 
The regularity statement of Proposition \ref{gt-laplacian} now gives $\zeta \in A^{k-1}[W^{2,2}]$. 
Thus, $\alpha - \eta = d(\zeta ) \in A^k[W^{1,2}]$, and since $\eta$ is smooth we get 
$\alpha \in A^k[W^{1,2}]$ as required. 
\end{proof}

This regularity serves the following purpose. 

\begin{corollary}\label{reg2}
If $\alpha \in A^k[W^{1,2}]$ and $\Delta_d (\alpha ) = 0$ distributionally, then 
$\alpha \in \mathcal{H}^k_g$.
\end{corollary}
\begin{proof}
Proposition \ref{gt-laplacian} applies, to give $\alpha \in A^k[W^{2,2}]$. 
Then Corollary \ref{laplacian-norm-squared} implies that $d\alpha = 0$ and $d^{\ast}\alpha = 0$, thus $\alpha \in \mathcal{H}^k_g$.
\end{proof}

\begin{proposition}
\label{reg22}
If $\alpha \in \mathcal{H}^k_g$ then $\alpha \in A^k[W^{2,2}]$ and 
$\Delta _d(\alpha ) = 0$. 
\end{proposition}
\begin{proof}
If $\alpha \in \mathcal{H}^k_g$ we have by definition that $\alpha \in A^k[L^2]$ and also that $d\alpha = 0$ and $d^{\ast}\alpha = 0$. Therefore, in fact $\alpha \in A^k[L^2,Dom(d)] \cap A^k[L^2,Dom(d^{\ast})]$. Proposition \ref{reg12} applies to give $\alpha \in A^k[W^{1,2}]$. Now, distributionally we have $\Delta _d \alpha = 0$, by part (2) of Proposition \ref{gt-laplacian} since $d(\alpha )=0$ and $d^{\ast}(\alpha ) = 0$. Therefore,
the regularity of part (3) of Proposition \ref{gt-laplacian} applies, to give $\alpha \in A^k[W^{2,2}]$. 
\end{proof}

\subsection{The Hodge $(p,q)$ components}

\begin{theorem}\label{laplaciancommutes}
On the $W^{2,p}\times W^{1,p}$ almost-hermitian and formally Kähler manifold $(X,g,J)$, define $\mathbf{J}$ as in section \ref{1.1}. Then, when acting on forms $\alpha \in A^k[W^{2,2}]$, $\Delta _d$ commutes with any polynomial in $\mathbf{J}$. In particular, $\Delta _d$ commutes with the projections $\Pi ^{p,q}$. 
\end{theorem}

\begin{proof}
One must check that the discussion of section \ref{1.1} holds for the $W^{2,p}\times W^{1,p}$ almost-hermitian and formally Kähler manifold $(X,g,J)$. The formal calculatory aspect of the proof is handled by Proposition \ref{DeltaJ} and Corollary \ref{DeltaJm}. We may verify that the multiplications involved in the commutator all make sense in our less regular setting, and that the Leibniz rules hold for the available level of regularity. To that end, note that if $\alpha \in A^k[W^{2,2}]$ then $\mathbf{J}^m \Delta _d \mathbf{J}^{m'}\alpha \in W^{1,p} \cdot W^{-1,q}$ for some $q>2$ (since $\mathbf{J}^{m'}\alpha$ is in $W^{1,q}$). There is a Sobolev multiplication from $W^{1,p} \cdot W^{-1,q}$ to $L^1$, see Section \ref{lessreg}. The Weitzenb{\"o}ck curvature operator $R$ is in $L^p$ and it can multiply elements of $W^{1,p}\cdot W^{2,2}$.

The two quantities 
$(\Delta_d - \nabla^* \nabla )$ and $R$ 
identified in the Weitzenb{\"o}ck formula of Proposition \ref{wf_thm}
depend on the derivatives of $g$ up to second order. These are defined almost everywhere for $g$ in $W^{2,p}$, and furthermore any second order jet can be seen as coming from a smooth metric somewhere, so Proposition \ref{wf_thm} holds for any second order jet of a metric $g$. Thus, the identification of \ref{wf_thm} holds almost everywhere, so it holds in $L^p$. 
The proofs of Proposition \ref{DeltaJ} and Corollary \ref{DeltaJm} then depend only on Leibniz rules and multiplications that are available as described in the previous paragraph. 
\end{proof}

Suppose $\alpha \in \mathcal{H}^k_g\otimes_{\mathbb{R}}\!\mathbb{C}$. Define $\alpha ^{p,q}:= \Pi ^{p,q} \alpha$. Recall that $J$, hence $\mathbf{J}$, are continuous so pointwise and continuously the $k$-forms decompose as a direct sum of forms of Hodge types $(p,q)$ with $p+q=k$. In other words $\sum _{p+q=k}\Pi ^{p,q}$ is the identity, giving $\alpha = \sum _{p+q=k}\alpha ^{p,q}$. 

By Proposition \ref{reg22}, $\alpha \in A^k[W^{2,2}]$ and $\Delta _d(\alpha ) = 0$. The projections are in $W^{1,p}$ so 
\[
\alpha ^{p,q} \in A^k[W^{1,p}\cdot W^{2,2}] \subset A^k[W^{1,2}].
\]
The Laplacian commutes with $\Pi ^{p,q}$ by Theorem \ref{laplaciancommutes}.  Therefore we get the distributional equation $\Delta _d\alpha ^{p,q}=0$. 
Corollary \ref{reg2} yields $\alpha ^{p,q} \in \mathcal{H}^k_g$. 

\begin{corollary}
\label{pqdecomp}
Let $\mathcal{H}_g^{p,q} \coloneqq \mathcal{H}^k_g \otimes_{\mathbb{R}}\!\mathbb{C} \cap A^{p,q}[L^2]$ be the space of harmonic forms of Hodge type $(p,q)$. Then 
\[
\mathcal{H}^k_g\otimes_{\mathbb{R}}\!\mathbb{C} = \bigoplus _{p+q=k} \mathcal{H}_g^{p,q} .
\]
\end{corollary}
\begin{proof}
The discussion above shows that the map 
\[
\bigoplus _{p+q=k} \mathcal{H}_g^{p,q} \rightarrow 
\mathcal{H}^k_g \otimes_{\mathbb{R}}\!\mathbb{C}
\]
is surjective. Pointwise, the $k$-forms decompose into their components of type $(p,q)$, so this map is injective. 
\end{proof}

Indeed, we obtain the following fact as a consequence of Corollary \ref{Hodge_thm} and Corollary \ref{pqdecomp}.

\begin{corollary}\label{Hodge_thm_Betti}
Let $X$ be a formally K{\"a}hler $W^{2,p} \times W^{1,p}$ hermitian manifold. Then, $b_{2k+1}(X)$ is even for all $k$.
\end{corollary}

\begin{proof}
Replace $\mathcal{H}^k$ with $\mathcal{H}^k_g$ and $\mathcal{H}^{p,q}$ with $\mathcal{H}^{p,q}_g$ in the proof of Theorem \ref{Betti}.
\end{proof}

\section{K{\"a}hler invariants at the threshold}\label{7}

\noindent
\emph{Proof of Theorem \ref{Betti_numbers_thm}}. For a contradiction, suppose that there exist an $a\geq 1$ and bounds $d,k,i_0$ such that for each $0<\epsilon_0 \leq a$ there is an $\epsilon_0$-acK manifold $(X_{\epsilon_0},g_{\epsilon_0},J_{\epsilon_0})$ with $(X_{\epsilon_0},g_{\epsilon_0}) \in \mathcal{M}(n,d,k,i_0)$ and odd Betti number $b_{2k+1}(X_{\epsilon_0})$ for some $k.$ Thus, we can define a sequence of $\frac{1}{i}$-acK manifolds $\{(X_{(\frac{1}{i})},g_{(\frac{1}{i})},J_{(\frac{1}{i})})\}_{i=1}^{\infty}$ with $(X_{(\frac{1}{i})},g_{(\frac{1}{i})}) \in \mathcal{M}(n,d,k,i_0)$, i.e.\ $\{(X_{(\frac{1}{i})},g_{(\frac{1}{i})},J_{(\frac{1}{i})})\}_{i=1}^{\infty} \subset \mathcal{A}(n,d,k,i_0,1)$.  By Theorem \ref{prepctness_thm}, we can extract a subsequence $\{(X_{(\frac{1}{i_l})},g_{(\frac{1}{i_l})},J_{(\frac{1}{i_l})})\}_{l=1}^{\infty}$ with $(g_{(\frac{1}{i_l})},J_{(\frac{1}{i_l})})$ converging in the weak-$W^{2,p}\times W^{1,p}$ topology to an acK manifold 
\linebreak[9]
$(X_{\infty},g_{\infty},J_{\infty})$ that is formally K{\"a}hler, i.e.\ $\|\nabla^{\infty} J_{\infty}\|_{g_\infty} \leq \liminf{\frac{1}{i}}=0.$ Hence $b_{2k+1}(X_{\infty})$ is even by Corollary \ref{Hodge_thm_Betti}. But diffeomorphism invariance of Betti numbers implies that $b_{2k+1}(X_{(\frac{1}{i_l})})$ is even as well, which is a contradiction.\QEDB

If we take a closer look at the proof of Theorem \ref{Betti_numbers_thm}, we see that in fact it proves a more general result, Theorem \ref{Thm-P}. Before enouncing the theorem let us discuss its proof. Theorem \ref{prepctness_thm} implies that any diffeomorphism-invariant property $P$ of compact formally K{\"a}hler manifolds, with $W^{2,p} \times W^{1,p}$ hermitian structure $(g,J)$, is also a property of compact $\epsilon$-acK manifolds for small enough $\epsilon>0$. In Theorem \ref{Betti_numbers_thm}, $P$ is the property that the odd Betti numbers are even. We essentially showed that property $P$ holds true for formally K{\"a}hler manifolds with $W^{2,p} \times W^{1,p}$ hermitian structure. Then, we used Theorem \ref{prepctness_thm} to argue by contradiction and conclude that there must be an $\epsilon$ so that any compact $\epsilon$-acK manifold also has property $P$. The contradiction is produced by the diffeomorphism invariance of $P$. The subtle point here is indeed the lesser than usual regularity. In fact, a property $P$ of the kind discussed above is quite abstract in practice. But most likely, this is not because such properties are scarce, rather it is because they need to be studied more intensively. The first place to look for such $P$ is K{\"a}hler geometry. We suspect that it is no coincidence that facts about truly K{\"a}hler manifolds, remain true up to a certain threshold of non-integrability.

\begin{theorem}\label{Thm-P}
Let $P$ be a property of compact formally K{\"a}hler manifolds with $W^{2,p} \times W^{1,p}$ hermitian structure that is invariant under diffeomorphism. Let $n>1$ be an integer, and $(X,g,J)$ be an almost-hermitian manifold of dimension $n.$ For any set of bounds $d,k,i_0,$ and any real number $a\geq 1,$ there is an $0<\epsilon_0 \leq a$ such that for any $(X,g) \in \mathcal{M}(n,d,k,i_0)$, if $(X,g,J)$ is $\epsilon_0$-acK, then $X$ enjoys property $P$ as well.
\end{theorem}

When the conclusion of this theorem holds, we say that the property $P$ {\em admits a positive Kähler threshold} and define the threshold to be the optimal $0 <\epsilon =\epsilon (P;n,d,k,i_0) \leq a$ with the stated property. 

\begin{theorem}\label{evenpos}
Let $P_{\rm even-pos}$ be the property of an $n$-dimensional K{\"a}hler manifold, that the even Betti numbers $b_{2k}$ for $0\leq k \leq n/2$ are strictly positive. This property admits a positive Kähler threshold: there is an $0< \epsilon_0 = \epsilon_0 (P_{\rm even-pos};n,d,k,i_0) \leq a$ such that if $(X,g,J)$ is an $\epsilon_0$-acK manifold then it has property $P_{\rm even-pos}$. 
\end{theorem}

\begin{proof}
We have now streamlined the above process: apply Theorem \ref{Thm-P} and observe that for a $W^{2,p}\times W^{1,p}$ hermitian formally Kähler manifold $(X,g,J)$, $\omega _J^k$ is harmonic, giving that $b_{2k}>0$ (c.f.\ Remark \ref{evenBetti} and Corollary \ref{Hodge_thm}). 
\end{proof}

\section{Future directions and questions}\label{8}

The following list groups in one place some interesting open problems inspired by our present work.

\begin{enumerate}

    \item \textbf{Properties $\mathbf{P}$.} Find other examples of properties $P$ that satisfy the hypotheses of Theorem \ref{Thm-P}. Future applications of this theorem could involve the Hard Lefschetz property, the vanishing of Borel volume regulators for local systems on compact Kähler manifolds \cite{Reznikov}, restrictions on fundamental groups, existence of mixed Hodge structures on homotopy groups, formality of the homotopy type, and structural results on the cohomology jumping loci.
    
    \item \textbf{Properties of almost-complex and/or symplectic manifolds.} Certain properties might be defined only for the case of almost-complex manifolds (for example, the degeneration of the Hodge to De Rham spectral sequence and equivalence of De Rham and Dolbeault cohomology) or symplectic manifolds. A version of Theorem \ref{Thm-P} should also hold for these properties. 
    
    \item \textbf{Bogomolov-Miyaoka-Yau and Bogomolov-Gieseker inequalities.} Find topological results for acK geometry that generalize such results for K{\"a}hler manifolds and that can be obtained by extending definitions and quantities to this new setting. 
    
    The famous Bogomolov, Bogomolov-Miyaoka-Yau and Bogomolov-Gieseker inequalities give some of the most refined information available on the topological properties of compact Kähler manifolds and smooth complex projective varieties. We would like to formulate these in a way that can enter into Theorem \ref{Thm-P}. This is not straightforward, since these inequalities depend on some geometrical hypothesis such as $K_X$ being ample, $X$ being of general type, or vector bundles or Higgs bundles being semistable. We therefore leave this area of questions as a direction for future study. But we would like to note that these properties have served as defining signposts for the study of topological  properties of complex manifolds and as such form a main source of motivation for the present paper. One possible approach is to use the concept of big representation of the fundamental group, however this is still somewhat geometric. Does it give a nef $K_X$?  

\item \textbf{Existence of a smooth Kähler structure.} The applications of Theorem \ref{Thm-P}, and the resulting various Kähler thresholds could be subsumed in a general statement saying that the limiting manifold has a smooth Kähler structure. 

\begin{conjecture}
\label{change-K}
For $p\gg 0$, a $W^{2,p}\times W^{1,p}$ formally Kähler manifold has a smooth structure compatible with $J$ s.t.\ nearby $\omega_J$, there is a symplectic form that is smooth for this structure, so that $X$ is $W^{2,p}$ diffeomorphic to a smooth Kähler manifold. 
\end{conjecture}

Two basic ingredients would need to enter into the proof. The first would be a Newlander-Nirenberg theorem-type statement that we formulate as its own conjecture: 

\begin{conjecture}
\label{NN}
The Newlander-Nirenberg theorem on integrability of the complex structure holds for a 
$W^{2,p}\times W^{1,p}$ formally Kähler manifold, leading to holomorphic coordinate charts and a new smooth structure in which $J$ becomes smooth while $g$ remains $W^{2,p}$. 
\end{conjecture}

At this point, the Kähler form $\omega_J$, which comes from $J$ and the Riemannian metric, remains $W^{2,p}$. To prove Conjecture \ref{change-K} one would then like to perturb $\omega_J$ to a nearby form $\omega'_J$ that is a smooth Kähler form for $(X,J)$ in the new smooth structure. 

It might be possible to use the Hodge decomposition theorem developed in this paper in order to do that.

We are wondering if the question formulated by Anderson \cite[Remark 2.3]{Anderson} about curvature concentration might not have an answer specific to the acK case: 

\begin{conjecture}\label{optimal}
For $\epsilon$-acK manifolds with $\epsilon \in (0,a]$, the $W^{2,p}$ harmonic radius \cite{Anderson_notes} is uniformly bounded below in terms of $(n,d,k,i_0,a)$. 
\end{conjecture}

Note that Conjecture \ref{optimal} would allow us to apply the technique of \cite{Anderson_notes} to get a weak (*) limit in $W^{2,\infty}\times W^{1,\infty}$.

\item \textbf{Computer assisted mathematics.} It will be interesting to launch computer-verified artificial intelligence formal proof developments of various parts of the arguments involved in Theorem \ref{Thm-P}. For instance, in section \ref{1.1} we have presented a proof of the classical Hodge decomposition theorem relying more on computation and less on conceptual elements than the usual proof. This could be well adapted to computer formalization (e.g.\ in {\sc Lean} or {\sc Rocq}) where the verification of series of rewrite statements in a calculation does not pose too much of a problem.

Some of the steps in a formalization procedure, such as the refinement of regularity hypotheses throughout the arguments, could provide a roadmap for many things to formalize in global analysis and differential geometry. This could notably include parts of 
Gilbarg-Trudinger \cite{GilbargTrudinger}. 

Along with the general question of formalization, there is also the question of how to make various parts of the argument more constructive, without the axiom of choice and in a finitary setting. Clearly, our current reliance on general compactness results will hinder this type of pursuit. This topic is therefore closely related to the search for a finitary expression of the theory, and the search for explicit bounds on the Kähler threshold. 

\item \textbf{Constructive and finitary statements.} The notion of $\epsilon$-acK manifold could serve to make precise the intuition that topological properties of Kähler manifolds are in fact finitary questions. For this, one would like to approximate the underlying Riemannian manifold by a triangulation structure, see for example Dodziuk \cite{DodziukCombinatorial}, or possibly as a point-cloud as discussed recently by Lê \cite{Le}.  
One would need to try giving a definition of the curvature being bounded needed to formulate a bounded geometry property. Then it would be necessary to have a theory of refinements of the point-cloud in order to get an approximate calculation of $\| \nabla J \|$ and formulate the $\epsilon$-acK condition. 

\item \textbf{Explicit bounds for the Kähler threshold.} In cases where the Kähler threshold is positive, such as Theorem \ref{evenpos}, we would like to get explicit lower bounds for the threshold 
$\epsilon (P;n,d,k,i_0,a)$. Our method of proof used a general compactness argument and extraction of subsequences, so it does not give any information on the explicit bound.

\end{enumerate}

\noindent
\textbf{Acknowledgements.} The research leading to these results has received funding from the European Research Council (ERC) under the European Union's Ninth Framework Programme Horizon Europe (ERC Synergy Project Malinca, Grant Agreement n.\ 101167526).

\end{document}